\documentclass[12pt]{article} 
\usepackage{amsmath, amscd, amssymb,latexsym, epsfig, color, mathrsfs, graphicx, xcolor, relsize }
\usepackage{tikz}
\usetikzlibrary{arrows,decorations.pathmorphing,backgrounds,positioning,fit,petri}
\tikzset{main node/.style={circle,draw,minimum size=.3cm,inner sep=0pt}, }
\usetikzlibrary{topaths,calc} 

\definecolor{cbgold}{RGB}{255,193,7}
\definecolor{cbblue}{RGB}{30,136,229}
\definecolor{cbpink}{RGB}{216,27,96} 
\definecolor{cbgreen}{RGB}{0,77,64} 

\newtheorem{thm}{Theorem}[section]
\newtheorem{prop}[thm]{Proposition}
\newtheorem{cor}[thm]{Corollary}
\newtheorem{defn}[thm]{Definition}

\newtheorem{rem}[thm]{Remark} 
\newtheorem{lem}[thm]{Lemma}
\newtheorem{example}[thm]{Example}

\newtheorem{rmk}[thm]{Remark}

 \def\Z{\mathbb{Z}} 
 
\def\Q{\mathbb{Q}}
\def\R{\mathbb{R}}
\def\C{\mathbb{C}}

\def\a{{\bf{a}}}

\def\e{{\bf{e}}}

\def\CC{\mathcal{C}} 
\def\H{\mathcal{H}}
\def\J{\mathcal{J}}

\def\FF{\mathcal{F}}

\def\G{\mathcal{G}}
\def\L{\mathcal{L}}

\def\KK{\mathcal{K}}
\def\A{\mathcal{A}}
\def\BG{\mathcal{BG}}
\def\SG{\mathcal{SG}}

\newcommand{\ef}{\mbox{effectively free}}

\newcommand{\triv}{\mbox{\upshape triv}\,}

\newcommand{\pr}{\pi}

\newcommand{\Field}{{\mathbb{F}}} 
\newcommand{\Rep}{{\mathcal{R}}}
\newcommand{\poly}{{polynomial }}
\newcommand{\possible}{{possible }}
\newcommand{\possiblee}{{possible}}

\title{Polymatroids are to finite groups as matroids are to finite fields }
\author{Ed Swartz\\ \small{Cornell University, Ithaca, NY}\\  \\Prairie Wentworth-Nice\\ \small{Cornell University, Ithaca, NY}\\  \\ Alexander Xue \\ \small{UCLA, Los Angeles, CA}}

\begin{document}
\maketitle

\begin{abstract}
Given a subgroup $\H$ of a  product of finite groups $\G = \displaystyle\prod^n_{i=1} \Gamma_i$ and $b>1,$ we define a polymatroid $P(\H,b).$  If all of the $\Gamma_i$ are isomorphic to $\Z/p\Z,~p$ a prime, and $b=p,$ then $P(\H,b)$ is the usual matroid associated to any $\Z/p\Z$-matrix whose row space equals $\H.$ In general, there are many ways in which the relationship between $P(\H,b)$ and $\H$ mirrors that of the relationship between a matroid and a  subspace of a finite vector space.  These include representability by excluded minors, the Crapo-Rota critical theorem, the existence of a concrete algebraic object representing the polymatroid dual of $P(\H,b),$ analogs of Greene's theorem and the MacWilliams identities when $\H$ is a group code over a nonabelian group, and a connection to the combinatorial Laplacian of a quotient space determined by $\G$ and $\H.$  We use the group Crapo-Rota critical theorem to demonstrate an extension to hypergraphs of the classical duality between  proper colorings and nowhere-zero flows on graphs.

\end{abstract}

\section{Introduction}

  It has been almost ninety years since Whitney began the theory of matroids \cite{Wh}.  Since then, the reach and interest in matroids has expanded so dramatically that it is no longer possible to give a reasonably sized  review of the literature involved.  A significant portion of that literature involves linear algebra over finite fields.   One of the most common settings for matroids starts with a matrix $A$ with coefficients in a finite field $\Field.$ Let $E$ be the columns of $A.$ Then $E$ is the ground set for a matroid $M[A]$. The rank function, $r_A,$  of $M[A]$ assigns to every subset $S$ of the columns of $A,$ the dimension of the span of the columns in $S$ in the column space of $A.$     An alternative approach is to start with a subspace $W$ of the  $\Field$-vector space $\Field^E.$  Then choose a matrix $A$ whose row space equals $W$ and set $M[W]$ to be the matroid $M[A].$  Since matrices with the same row space have the same column dependence relations, $M[W]$ is well defined.  It was this point of view, applied to elements of the complex Grassmannian and its moment map, that led Gelfand, Goresky, MacPherson, and Serganova to rediscover  Edmonds' matroid polytopes \cite{GGMS}.

Matroid polytopes are a special case of polymatroid polytopes.  The latter were introduced by Edmonds   over fifty years ago as a polyhedral generalization of the matroid greedy algorithm, which itself was a generalization of the greedy algorithm for  finding maximal weight spanning trees  in a graph \cite{E}.  Since Postnikov's influential paper \cite{Po}, polymatroids have also been known as `generalized permutahedra'. Investigations from that perspective tend to emphasize the polyhedral and algebraic geometry of Edmonds' original polytopal presentation.  We do not examine any connection between this approach and ours.  However, it seems reasonable to hope, indeed expect, that the two lines of research will eventually intersect.

Our main objective is to demonstrate that the relationship between subspaces $W$ of a finite vector space $\Field^E$ and their associated matroids $M[W],$   is similar to a previously unrecognized relationship between polymatroids and subgroups of a product of finite groups.   Given such a $W,$ let $A$ be a matrix whose row space equals $W.$ For $S \subseteq E$ let $\pr_S(W)$ be the projection of $W$ to $\Field^S.$ Since $\Field$ is finite,  $r_A(S)$ equals   $\log_{|\Field|} |\pr_S(W)|.$  We can think of $r_A$ as a CT-scan of $W.$ Just as a CT-scan measures how much mass interferes with x-rays from a variety of projections, $r_A$ records how much of $W$ lies in a variety of projections.  

  Let us replace $\Field^E$ with a  product of finite groups $\G = \displaystyle\prod_{x \in E} \Gamma_x,$  and replace $W$ with a subgroup $\H$ of $\G.$  Fix a real number $b >1.$  Now our figurative CT-scan of $\H$ consists of the set  of logarithms base $b$ of the cardinality of the projections of $\H$ to $\displaystyle\prod_{x \in S} \Gamma_x$ for all $S \subseteq E.$  This collection of data forms a polymatroid (Theorem \ref{the polymatroid}) which we denote by $P(\H,b)$ or $P(\H)$ for brevity.   
  
  An algebraic characterization of subgroups of a product of two groups goes back to  Goursat's Lemma \cite{Go}.  Bauer, Sen and Zvengrowski established algebraic criteria for subgroups of an arbitrary finite product of groups \cite{BSZ}.  Our purpose here is very different.  
   Most of the paper is devoted to showing several ways in which $P(\H)$ interacts with $\H$ in a fashion similar to the way $M[W]$ interacts with $W.$   We do this for several areas where matroids and finite fields meet. The last two are new even for matroids.   

\begin{itemize}
\item The class of matroids of the form $M[W]$ as $W$ ranges over all subspaces of finite $\Field$-vector spaces is closed under minors.  The class of polymatroids of the form $P[\H]$ as $\H$ ranges over all subgroups of finite products of a fixed finite group $\Gamma$ is closed under minors. (Section \ref{Representability})

\item Theorem \ref{group Crapo-Rota}  is a Crapo-Rota critical formula for finite groups which is virtually identical to its finite field counterpart.  We use this theorem to give a uniform critical theorem proof of Helgason's formula for the number of proper colorings of a hypergraph in terms of the characteristic polynomial of an associated polymatroid.  We are also led by Theorem \ref{group Crapo-Rota} to a hypergraph analog of the classical duality between proper colorings and nowhere-zero flows on graphs. (Theorem \ref{hypergraph flows}) 

\item As Whitney showed in \cite{Wh}, if $W^\perp$ is the orthogonal complement of $W$ in $\Field^E,$  then $M[W^\perp]$ is the matroid dual of $M[W].$ We define  a representation-theoretic object, denoted by $\Rep(\H),$ which is a realization of the polymatroid dual of $P(\H)$.  Many duality properties of $M[W^\perp]$ in the matroid setting are mimicked by $\Rep(\H)$ in the polymatroid setting. (Section \ref{duality for realized polymatroids})

\item Viewed as a linear code, $W^\perp$ is the dual code of $W$ and satisfies Greene's theorem which relates the Tutte polynomial of $M[W]$ to the  weight enumerator of $W.$ This leads to a simple proof of the famous MacWilliams identity between the code weight enumerator of $W$ and that of its dual code.  When viewed as a group code, $\H$ and $\Rep(\H)$ satisfy  an analog of Greene's  theorem which leads to a  MacWilliams identity for $\H$ and $\Rep(\H).$ When $\H$ is a subgroup of $(\Z/p\Z)^n,~p$ a prime, this is the usual MacWilliams identity.  There is also a MacWilliams identity relating the complete weight enumerators of $\H$ and $\Rep(\H)$ \cite{We}. If $\H \le \Gamma^E$ with $\Gamma$ an abelian group, then the MacWilliams identity for the complete weight enumerator of $\H$ and $\Rep(\H)$ is the usual one. (Section 10) 

\item The elementary abelian $2$-group $(\Z/2\Z)^n$ acts coordinate-wise on $\partial CP_n,$ the boundary of the $n$-dimensional cross polytope.  Let $W$ be a subgroup of $(\Z/2\Z)^n,$ so  $W$ inherits this action.  Denote the quotient space $\partial CP_n/W$ by $X_W.$  If $M[W]$ has no coloops, then the eigenvalues of the top-dimensional combinatorial Laplacian applied to $X_W$ contains the same information as the code weight enumerator of the dual code $W^\perp$ (Theorem \ref{binary eigenvalues}).   A statement which reduces to this result holds for $\H$ a subgroup of $\G$ with $\Rep(\H)$ substituting for $W^\perp$ and a  quotient space $X_\H$ derived from $\G$ and $\H$ replacing $\partial CP_n.$ (Theorem \ref{group eigenvalues})

 \item The dimension of $H^{n-1}(X_W; \Q)$ equals the characteristic polynomial of $M[W]^\ast$ evaluated at two.  The characteristic polynomial of $P(\H,b)^\ast$ evaluated at $b$ equals $H^{n-1}(X_\H; \C).$ (Theorem \ref{char poly as top dimension}) 

\end{itemize}

A secondary, perhaps more subtle, objective of this paper is  the simplicity of the proof of Theorem \ref{the polymatroid}.   Anytime subobjects of product spaces are under consideration, especially in an algebraic setting, there is a possibility that an interesting polymatroid is hiding in plain sight.  For an example, see \cite{CCLMZ}. 

By using rank functions which agree with the usual matroid rank functions when the groups are the additive groups of vector subspaces, our rank functions are frequently not integers.  While this shines a clear light on the similarities between the polymatroid and matroid settings, it does have the effect of, to a certain extent, hiding  the fundamental integral nature of the objects involved.   In principle, using something along the lines of the R\'edei function of \cite{KMR} could keep everything in the integers.  Whether or not a purely integral approach leads to better insight or understanding remains to be seen.  

\section{Matroids and polymatroids} \label{Notations and definitions}

In this section we introduce several basic matroid and polymatroid notions. Other definitions and notation for matroids, polymatroids and other subjects will appear later as they are needed.

A {\bf polymatroid} is a pair, $P=(E,r),$ where $E$ is a  finite ground set and $r$ is a rank function $r:2^E \to \R_{\ge 0}$ such that 
\begin{enumerate}
\item[P1.] $r(\emptyset) = 0.$ (normalization)
\item[P2.] If $S \subseteq T,$ then $r(S) \le r(T).$ (monotonicity)
\item[P3.] \label{submodularity} For all $S,T \subseteq E,~r(S) + r(T) \ge r(S \cup T) + r(S \cap T).$ (submodularity)
\end{enumerate}
\noindent A cryptomorphic definition is obtained by replacing P3 with the axiom of {\bf diminishing returns}:  
\begin{enumerate}
\item[P3'.]  For all $S \subseteq T$ and $x \in E,$
 $r(S \cup \{x\}) - r(S) \ge r(T \cup \{x\}) - r(T).$
\end {enumerate}

Two polymatroids $P=(E,r)$ and $P'=(E', r')$  are {\bf isomorphic}, written $P \cong P',$  if there is a rank-preserving bijection from $E$ to $E'.$ 
Throughout the paper $P=(E,r)$ is a polymatroid, $n = |E|$, and $[n]=\{1, \dots, n\}.$   Often $P$ (and $E$ and $r$) will be decorated in various ways in order to distinguish explicitly which polymatroid is under consideration.   

A polymatroid is {\bf subcardinal} if for all $S \subseteq E, r(S) \le |S|.$  A subcardinal polymatroid which is integer valued is a {\bf matroid}.  From here on $M$ will be a matroid, frequently decorated to specify which matroid.    Matroids were first defined by Whitney \cite{Wh}.  For a more complete early history of matroids see \cite{Ku}. This includes  a frequently forgotten contemporaneous equivalent definition by Nakasawa \cite{Na1}, \cite{Na2}, \cite{Na3}.   Polymatroids were first introduced by Edmonds as a polytopal generalization of the greedy algorithm for matroids \cite{E}.  

The prototypical example of a matroid starts with an $m \times n$  matrix $A$ with coefficients in a field $\Field.$   The elements of the ground set  $E$ are the columns of $A$ in $\Field^m.$ Distinct columns of $A$ are distinct elements of $E,$ even if as column vectors in $\Field^m$ they are identical.  The rank $r_A(S)$  of a subset $S$ of $E$ is the dimension of the span of the column vectors in $S.$ A matroid is {\bf representable over $\Field$} if it is isomorphic to $(E,r_A)$ for some $A.$ In that case $A$ is a {\bf realization} of the matroid.  Usually we use $M[A]$ to denote $(E, r_A).$ 

An alternative, but equivalent, approach to matroid representability is through subspaces.  Suppose $W$ is a subspace of $\Field^n.$ If $A$ and $B$ are two matrices whose row spaces $R(A)$ and $R(B)$ are equal to $W,$ then $M[A]$ is isomorphic to $M[B]$.  Indeed,  $A$ and $B$ are connected through a series of elementary row operations (where inserting or removing a row of zeros is included as an elementary row operation) and row operations do not alter column independence relations.  In this situation, we will use $M[W]$ to stand for this isomorphism class of matroids.  As with matrices, if $M$ is a matroid and $M$ is isomorphic to $M[W],$ then $M$ is representable over $\Field,$ and $W$ is a realization of $M.$ 

Another motivating example for matroids comes from graphs.  Given a graph $G,$ let $E(G)$ be the edge set of $G.$  The rank of a subset  $S$ of edges of $G$ is assigned the rank $r_G(S)$ equal to the cardinality of the largest acyclic subset of edges in $S.$ Then $(E(G), r_G)$ is the {\bf cycle matroid of $G$} which we denote by $M[G].$ In fact, $M[G]$ is equal to $M[A]$ where $A$ is the following adjacency matrix.   First choose an orientation of the edges of $G.$ The columns of $A$ are indexed by edges  of $G$ and the rows are indexed by the vertices of $G.$  The entry of $A$ corresponding to an edge $x$ and vertex $v$ is: $1,$ if $v$ is the head of $x,~-1$ if $v$ is the tail of $x,$ and $0$ if $x$ and $v$ are not incident or $x$ is a loop.  Then $A$ is a realization of $M[G]$ over every field.

A (usually not subcardinal) polymatroid whose rank function is integer-valued is often referred to as a {\bf discrete} or {\bf integer} polymatroid.  Some authors define polymatroids to be integer valued.   Another common nomenclature is {\bf $k$-polymatroids} for discrete polymatroids such that $k \ge \max_{x \in E} r(\{x\}).$  One way to produce integer polymatroids is to start with a matroid $M=(E,r)$ and let $\{T_1, \dots, T_n\}$ be a multiset of subsets of $E.$  For $S \subseteq [n]$ define $r'(S)$ to be the rank in $M$ of $\displaystyle\cup_{i \in S} T_i.$  Then $P=([n], r')$ is a polymatroid.  In fact, all integer polymatroids are isomorphic to such a construction 
\cite[Theorem 11.1.9]{Ox}. 
  In this case, {\bf $P$ is constructed from} $M$ and $\{T_1, \dots, T_n\}.$  When there is no cause for confusion, the reference to $\{T_1, \dots, T_n\}$ is omitted. 
An integer polymatroid $P$ is {\bf representable} over a field $\Field$ if there exists a matrix $A$ with coefficients in $\Field$ such that $P$ is constructed from $M[A].$   In this case $A$ (and $\{T_1, \dots, T_n\}$) is a {\bf realization} of $P.$  In a fashion similar to matroids, representability of polymatroids over a field $\Field$ can be approached using a subspace $W$ of $\Field^m,$  where $m = \displaystyle\sum^n_{i=1} r(\{i\}),$ and choosing a basis of $W$ to be the rows of $A.$

   As for matroids, a {\bf flat} of a polymatroid $P$ is a subset $S$ of $E$ such that for all $x$ not in $S,  ~r(S \cup \{x\}) >r(S).$    The poset of the flats of a polymatroid ordered by inclusion forms a lattice, called the {\bf lattice of flats of $P$},  with meet being intersection.   In fact, any finite lattice is the lattice of flats of a polymatroid.  One can even insist that the polymatroid be one constructed from $M[K_n],$ where $K_n$ is the complete graph on $n$ vertices.  See \cite{RS} for a description and history of this result.  We will use $\L_P$ to represent the lattice of flats of the polymatroid $P.$

   There is a polymatroid analog of the characteristic polynomial of a matroid.  However, unless the polymatroid is discrete, it is generally  not a polynomial.  Despite this, we will still call it the {\bf characteristic \poly} of a polymatroid.  It is defined completely analogously to the definition of the characteristic polynomial of a matroid.  
   $$\mathlarger{\chi}_P(t) = \displaystyle\sum_{S \subseteq E} (-1)^{|S|} t^{r(E)-r(S)}.$$
   
   The characteristic \poly  of a polymatroid $P$ has many properties in common with its matroid counterpart.  For instance, the characteristic \poly of $P$ is the zero function if $P$ has an element $x \in E$ such that $r(\{x\})  = 0.$ Such an element is called a {\bf loop}. There is also a formula for $\mathlarger{\chi}_P$ based on the M\"obius function $\mu_P$ of $\L_P$ that is identical to the usual matroid definition.   
   
   \begin{prop} \label{mobius char poly}
   If $P$ has a loop, then $\mathlarger{\chi}_P \equiv 0.$ If $P$ does not contain a loop, then 
   $$\mathlarger{\chi}_P(t) = \displaystyle\sum_{S \in \L_P} \mu_{\L_P}(\emptyset, S) t^{r(E)-r(S)}.$$
   \end{prop}
   \begin{proof}
   The same proof as for matroids works here. See, for instance, \cite[Proposition 7.2.1]{Za}.
   \end{proof}

   \section{From finite groups to polymatroids} \label{groups to polymatroids}
The notation in this paragraph will be used in the remainder of the paper.  For all $x \in E,~\Gamma_x$ is a finite group.  In general, $\Gamma_x$ may or may not be isomorphic to $\Gamma_y.$   In addition, $\G = \displaystyle\prod_{x \in E} \Gamma_x,$ and $\H$ is a subgroup of $\G.$  For $S \subseteq E,~\G_
S = \displaystyle\prod_{x \in S} \Gamma_x$ and $\pr_S$ is the projection map $\pr_S: \G \to \G_S.$ We denote the image  $\pr_S(\H)$ by $\H_S.$  By convention,  $\G_\emptyset$ and $\H_\emptyset$ are the trivial groups. 

The  connection between a subgroup $\H \le \G$ and polymatroids is through a rank function.  Choose $b > 1, b \in \R.$ For all $S \subseteq E$ define
\begin{equation} \label{rank function}
  r_{\H,b}(S) \equiv \log_b |\H_S|.
\end{equation}
\noindent  Different choices of $b$ lead to rank functions which differ by a constant.  Indeed, 
 $$r_{\H,b} = \frac{\ln b'}{\ln b}r_{\H,b'}.$$
\begin{thm}  \label{the polymatroid}
 The pair $(E, r_{\H,b})$ is a  polymatroid.
\end{thm}

\begin{proof}
 Normality and monotonicity are evident, so let us check the law of diminishing returns for $S \subseteq T \subseteq E.$ If $x \in T$ then, (3') is automatically satisfied, so assume that $x \notin T.$  In general, for any $S \subseteq E,$
 
 \begin{equation} \label{one more} r_{\H,b}(S \cup \{x\}) - r_{\H,b}(S)  = \log_b \frac{ |\H_{S \cup \{x\}}|}{|\H_S|} = \log_b |\ker \pr_S:\H_{S \cup \{x\}} \to \H_S|.
 \end{equation}
 
 Let $\KK_S$ be $\ker \pr_S:\H_{S \cup \{x\}} \to \H_S$ and $\KK_T$ be $\ker \pr_T:\H_{T \cup \{x\}} \to \H_T.$ Since $S \subseteq T,~\pr_{S \cup \{x\}}: \G_{T \cup \{x\}} \to \G_{S \cup \{x\}}$ becomes an injection $\pr_{S \cup \{x\}}: \KK_T \to \KK_S$  when restricted to $\KK_T.$ Hence $|\KK_T| \le |\KK_S|$ as required.  
\end{proof}

\noindent We will use $P(\H,b)$ to denote the polymatroid $(E,r_{\H,b}).$  Most of the time  $b$ is understood or not relevant, so we suppress it and use $P(\H).$

\begin{example} \label{Z/pZ}
Let $E = [n]$ and for all $1 \le i \le n,~\Gamma_i = \Z/p\Z,~p$ a fixed prime, and $\G = (\Z/p\Z)^n. $  A subgroup $\H$ of $\G$ is a also a subspace of the vector space $\G$ over $\Field=\Z/p\Z.$ Then, with $b=p,~P(\H)$  is a matroid and is isomorphic to $M[\H].$ 
\end{example}

In principle, the definition of $r_{\H,b}$ makes sense for any collection of  sets $\{\Lambda_x\}$ indexed by $E$ and finite subset $\H$ of $\prod_{x \in E} {\Lambda}_x.$  It is not clear what sort of structure is needed on the $\H$ and the $\Lambda_x$ to guarantee that $r_{\H,b}$ is submodular.  What is clear is that some structure is needed. 

\begin{example} \label{not submodular}
Let $\Lambda_1=\Lambda_2=\Lambda_3 = \{1,2\}, E=\{1,2,3\},$ and $L \subseteq \Lambda_1 \times \Lambda_2 \times \Lambda_3$ be $$L = \{(2,2,2),(2,1,2),(1,1,2),(1,1,1),(2,1,1)\}.$$
\noindent Then $|L_{\{1,2\}}| \cdot |L_{\{2,3\}}|=3\cdot3 <|L_{\{1,2,3\}}| \cdot |L_{\{2\}}|=5\cdot2.$ Hence $r_{L,b}$ is not submodular for all $b.$

\end{example}

\section{Representability} \label{Representability}
Throughout Sections \ref{Representability}, \ref{Matroid representability} and \ref{Polymatroid representability}  we will be assuming that all of the $\Gamma_x$ are isomorphic to a fixed finite group $\Gamma$ and that $b = |\Gamma|.$ Consequently we will suppress all appearances  of $b$ from the notation.   In this section we introduce the definition of a polymatroid being representable over a finite group that  is analogous to  matroid representability over a field.

\begin{defn}
A polymatroid $P=(E,r)$ is {\bf representable over a finite group $\Gamma$} if there exists a finite set $E'$ and subgroup $\H$ of $ \Gamma^{E'}$ such that $P$ is isomorphic to $P(\H).$  In this case, $\H$ is a {\bf $\Gamma$-realization} of $P.$
\end{defn}

 If $\Gamma = \Z/p\Z,~p$ a prime, then  Example \ref{Z/pZ} shows that  a polymatroid $P$ is representable over the group $\Z/p\Z$ if and only if it is representable over the field $\Z/p\Z.$  In both cases, $P$ must actually be a matroid.    Another way that polymatroid realizability over a group mirrors matroid realizability over a field  is that it is closed under deletion, contraction and hence minors.   

Given  $S \subseteq E$ the {\bf deletion} of $S$ from $P$ is the polymatroid $P \setminus S$ with ground set $E-S$ and rank function equal to $r$ restricted to the subsets of $E-S.$   If $\H$ is a $\Gamma$-realization of $P,$ then $\H_{E-S}$ is a $\Gamma$-realization of $P \setminus S.$ 

Another polymatroid with ground set $E-S$ is $P/S,$ the {\bf contraction} of $P$ by $S.$  For $T \subseteq E-S,$ the rank function of $P/S$ is

$$r_{P/S}(T) = r(S \cup T) -r(S).$$

\noindent If $\H$ is a realization of $P$, then $\KK = \pr_{E-S} ( \ker \pr_S)$ is a realization of $P/S.$ To see this,  let $T \subseteq E- S.$   Then $\pr_T$ is an isomorphism between  $\KK_T $ and $\ker \pr_S: \H_{S \cup T} \to \H_S.$  Since $\pr_S$ is a surjection, 
$$|\KK_T|=|\ker \pr_S:\H_{S \cup T} \to \H_S| = \frac{|\H_{S \cup T}|} {|\H_S|}.$$

Minors of polymatroids are defined as they are for matroids.  A {\bf minor} of a polymatroid $P$ is a polymatroid $Q$ that can be obtained from $P$ by a sequence of deletions and/or contractions.  Let $\CC(\Gamma)$ be the isomorphism classes of polymatroids representable over $\Gamma.$ The above discussion shows that $\CC(\Gamma)$ is closed under minors.  Let $\CC$ be a class of polymatroids. We say $\CC$ is {\bf minor closed} if $\CC$ is  closed under isomorphism and minors.    An {\bf excluded minor} of $\CC$ is a polymatroid $P$ such that $P \notin \CC,$ but every deletion and contraction of $P$ is in $\CC.$  Minor closed classes of polymatroids can be described by listing their excluded minors.  Before discussing excluded minors for $\CC(\Gamma)$ we consider excluded minors for {\it matroids} representable over $\Gamma.$  

\section{Matroid representability} \label{Matroid representability}

In this section we restrict our attention to representability of matroids over a finite group $\Gamma.$ We denote the class of matroids representable over $\Gamma$ by $\CC^{Ma}(\Gamma).$ For instance, the polymatroid $P=(\{1,2,3,4\},r)$ with $r(S) = \min (1, |S|/2 )$ is an excluded minor for $\CC(\Z/4\Z),$ but is not a matroid and hence not an excluded minor for $\CC^{Ma}(\Z/4\Z).$ 

Given nonnegative integers $d \le n,$ the {\bf uniform matroid $U_{d,n}$} is the matroid with ground set $[n]$  and rank function $r(S) = \min (d, |S|).$  Two distinct elements $\{x,y\}$ in a polymatroid $P=(E,r)$ are {\bf parallel} if $0 < r(\{x\}) = r(\{y\}) = r(\{x,y\}).$  Diminishing returns (P3') implies that the relation `is parallel to' is an equivalence relation and the equivalence classes are called the {\bf parallel classes} of the polymatroid.  The {\bf simplification} of a matroid is the matroid obtained by deleting all loops and all but one element of every parallel class.

The unique excluded minor for $\CC^{Ma}(\Gamma)$ when $\Gamma$ is a nonabelian group is $U_{2,3}.$ Specifically, we have the following result.  

\begin{thm}
Let $\Gamma$ be a nonabelian group.  Then the following are equivalent for a matroid $M.$
\begin{enumerate}
  \item $M \in \CC^{Ma}(\Gamma)$.
  \item $M$ does not have $U_{2,3}$ as a minor.
  \item The simplification of $M$ is $U_{n,n}$ for some $n,$ or $M$ consists entirely of loops. 
  \item $M$ is representable over every finite group. 
\end{enumerate}
\end{thm}

\begin{proof}

(1) implies (2):  To avoid confusion with too many numerals, we use $E = \{x,y,z\}$ as the ground set of $U_{2,3}.$ Suppose $\H$ is a subgroup of $\Gamma_x \times \Gamma_y \times \Gamma_z$ so that $P(\H)$ is isomorphic to $U_{2,3}.$  Let $\gamma$ and $\gamma'$ be elements of $\Gamma$ which do not commute. Since $r(\{x,y\}) = 2 = r(\{x,y,z\})$ we know that $\H_{\{x,y\}}$ is $\Gamma_x \times \Gamma_y$ and for every ordered pair $(\gamma_x, \gamma_y) \in \Gamma_x \times \Gamma_y$ there exists a unique $\gamma_z \in \Gamma_z$ such that $(\gamma_x, \gamma_y, \gamma_z) \in \H.$  Applying the same reasoning to $(1_\Gamma, \gamma) \in \Gamma_x \times \Gamma_z$ and $(1_\Gamma, \gamma') \in \Gamma_y \times \Gamma_z,$  there exist  triples $(1_\Gamma, \gamma_y, \gamma)$ and $(\gamma_x, 1_\Gamma, \gamma')$ in $\H.$ But this implies that $(\gamma_x, \gamma_y, \gamma \cdot \gamma')$ and $(\gamma_x, \gamma_y, \gamma' \cdot \gamma)$ are both in $\H,$ a contradiction. 

(2) implies (3) is a  matroid exercise.

(3) implies (4):  A set of $n$ loops can be realized via $\H$ equal to the trivial subgroup of $\G= \Gamma^n.$ If we let $\H = \G = \Gamma^n,$ then $([n],r_\H)$ is isomorphic to $U_{n,n}.$ Now suppose $\H \subseteq \G = \Gamma^E$ realizes a matroid $M.$  To realize $M$ with a loop added, let $E'$ be $E$ with a new element $y$ included. For all $h \in \H,$ define $h' \in \Gamma^{E'}$ by $h'_x = h_x$ for $ x \in E$ and $h'_y = 1_\Gamma.$ Then $\H' = \{h':h \in \H\}$ realizes $M$ with a loop added.  Similarly, a new element parallel to $x \in E$  can be adjoined to the matroid represented by $\H$ by introducing a new coordinate $y$ and setting it equal to the $x$ coordinate for all $h \in \H.$ 

(4) implies (1) is immediate. 

\end{proof}

Determining the excluded minors for matroids representable over $\Z/p\Z$ for $p$ a prime has been a critical motivating problem for decades.  See for instance, \cite{Ox},  and the announcement of a proof of Rota's conjecture that there are a finite number of excluded minors for matroids realizable over any fixed finite field \cite{GGW}.  

A {\bf totally unimodular} matrix is a real matrix all of whose subdeterminants are $-1, 0$ or $1.$  In particular, all of the entries of the matrix are $-1,0$ or $1.$  A matroid $M$ is {\bf regular} if there exists a totally unimodular matrix $A$ such that $M$ is isomorphic to $M[A].$   The {\bf Fano} matroid is the matroid realized over $\Z/2\Z$ by the matrix whose columns are the seven nonzero vectors of $(\Z/2\Z)^3.$ The dual of the Fano is the matroid realized over $\Z/2\Z$ by the subspace of $(\Z/2\Z)^7$ which is orthogonal to the row space of the Fano. See Section \ref{Polymatroid duality} for an explanation for the term dual.  The following excluded minor characterization of regular matroids is originally due to Tutte. 

\begin{thm} \cite{Tu} 
For a matroid $M$ the following are equivalent.
\begin{itemize}
  \item $M$ is a regular matroid.
  \item $M$ does not have $U_{2,4},$ the Fano or the dual of the Fano as a minor. 
  \item $M$ is representable over the field $\Z/2\Z$ and a field $\Field$ whose characteristic is not two. 
  \item $M$ is representable over all fields.
\end{itemize}
\end{thm}

  Here are two results for non prime cyclic groups whose proofs will appear in \cite{We2}.  

\begin{thm}  \label{matroid Z2n}
If $n$ is even and greater than two, then $M$ is representable over $\Z/n\Z$ if and only if $M$ is regular. 
\end{thm}

\begin{thm}  \label{matroid cyclic}
  If $n=ml,~m$ and $l$ relatively prime, then a matroid $M$ is representable over $\Z/n\Z$ if and only if it is representable over $\Z/m\Z$ and $\Z/ l\Z.$ 
\end{thm}

\noindent The following two theorems follow immediately from  Theorem \ref{matroid cyclic} and the main results of \cite{Whitt2}.

\begin{thm}
A matroid $M$ is representable over $\Z/3p\Z,~p$ an odd prime congruent to $2 \mod 3,$ if and only if $M$ is realized by a rational-valued matroid all of whose nonzero subdeterminants are of the form $\pm 2^j, j \in \Z.$
\end{thm} 

\begin{thm}
A matroid $M$ is representable over $\Z/3p\Z,~p$ a prime congruent to $1 \mod 3,$ if and only if $M$ is realized by a complex-valued matroid all of whose nonzero subdeterminants are  sixth-roots of unity. 
\end{thm}

\section{Polymatroid representability} \label{Polymatroid representability}

As in the last section, $\Gamma$ is a fixed finite group and $\G = \Gamma^n.$   Several of the proofs and  further discussion will appear in \cite{We2}.

Does it make sense to ask if $\Gamma$-representable polymatroids have a finite list of excluded minors?  One trivial answer is `no'.  Any polymatroid with exactly one element whose rank is a positive real number $s$ less than one is a minimal excluded minor for $\Gamma$-representability whenever $s^{|\Gamma|}$ is not an integer which divides $|\Gamma|.$ However, this does not capture the same point of view as matroids, even when $\Gamma = \Z/p\Z$ and $\Gamma$-representability over polymatroids is the same as $\Gamma$-representability for matroids.  One possible approach of describing polymatroid representability over $\Gamma$ in a way that does mimic matroid representability is the following.  

\begin{defn}
 A polymatroid $P=(E,r)$ is {\bf $\Gamma$-\possible} if for all $S  \subseteq E,~|\Gamma|^{r(S)}$ is an integer which divides $|\Gamma|^{|S|}.$  
\end{defn}

\begin{rmk}
Stronger requirements could reasonably be imposed on the definition of $\Gamma$-possible polymatroids.  For instance, since $\CC(\Gamma)$ is closed under contraction, if $P \in \CC(\Gamma),$ then for all $S \subseteq T,~|\Gamma|^{r(T-S)}$ divides  $|\Gamma|^{|T-S|}.$
\end{rmk}

  By restricting our attention to $\Gamma$-\possible polymatroids it is now possible for there to be a finite list of excluded minors for $\Gamma.$ Matroids are $\Gamma$-\possible   for any $\Gamma.$ More importantly, if $P$ is $\Gamma$-representable, then $P$ is $\Gamma$-\possiblee.  If $|\Gamma|$ is a prime power $p^k,$ then the class of polymatroids which are $\Gamma$-\possible is the class of all subcardinal polymatroids $P$ such that all ranks are rationals of the form $j/k,~j \in \Z_{\ge 0}.$  In this case, $P$ is $\Gamma$-\possible if and only if  $kP = (E, kr)$ is a $k$-polymatroid.   The particular case of representability over  an elementary abelian $p$-{\em group}, $\Gamma = (\Z/p\Z)^k,$ has already been considered under the guise of $k$-polymatroid representability over the {\em field} $\Z/p\Z.$ 

\begin{prop}
Let $p$ be a prime.  Then a $k$-polymatroid $P$ is representable over the {\emph{field}} $(\Z/p\Z)$ if and only if $(1/k)P$ is representable over the {\emph{group}} $(\Z/p\Z)^k.$
\end{prop}

In \cite{OSW} Oxley, Semple and Whittle prove that when $k \ge 2$ there are infinitely many excluded minors for $k$-polymatroids over any prime field $\Z/p\Z.$  Combining this with the previous proposition gives us the following result. 

\begin{thm} \label{infinitely many}
  Let $p$ be a prime and $k \ge 2.$ Then $(\Z/p\Z)^k$ has infinitely many excluded minors among $(\Z/p\Z)^k$-\possible polymatroids. 
\end{thm}

The classical notion of equivalence of realizations of a matroid over a field can be extended to realizations of polymatroids over finite groups.  Two  matrices  $A$ and $B$ with coefficients in a field $\Field$ are {\bf equivalent} realizations of a matroid $M$ if $M \cong M[A] \cong M[B]$ and $A$ can be obtained from $B$ by a sequence of one or more of the following operations. 

\begin{itemize}
 \item A permutation of the columns of the matrix.
 \item Multiplying a column by a nonzero scalar.
 \item An elementary row operation. (Including insertion or removal of a row of zeros.)
 \item Applying a field automorphism $\theta: \Field \to \Field$ to every entry of the matrix. 
\end{itemize}

Since elementary row operations do not change row spaces,  equivalence of realizations of a matroid by subspaces can be described as follows.  Two subspaces $U$ and $V$ of $\Field^E$ are equivalent realizations of a matroid $M$ if $M \cong M[U] \cong M[V]$ and $U$ can be obtained from $V$ by a sequence of one or more of the following operations.

\begin{itemize}
\item A permutation of the coordinates.
\item Fix $x \in E$ and multiply every $x$-coordinate of $v \in V$ by a fixed nonzero scalar $c_x \in \Field.$
\item Apply a fixed field automorphism $\theta: \Field \to \Field$ to all coordinates of all $v \in V.$  
\end{itemize} 

Our approach to equivalence of realizations of polymatroids by groups can be reduced to two operations.  Suppose $\H$ and $\H'$ are two realizations of $P.$ Then $\H$ and $\H'$ are equivalent realizations if $\H'$ can be obtained from $\H$  by a sequence of the following two operations.

\begin{itemize}
\item A permutation of the coordinates.
\item Apply a group automorphism $\theta: \Gamma \to \Gamma$ to one coordinate.  
\end{itemize}

When $p$ is a prime there are no nontrivial field automorphisms of $\Z/p\Z$ and all group automorphisms of $\Z/p\Z$ are multiplication by a nonzero scalar.  Hence the definition of equivalent realizations of matroids realizable over $\Z/p\Z$ as a field and a group coincide.  One of the basic facts of matroid theory is that all subspace realizations of matroids over $\Z/2\Z$ are equivalent. The same holds for $\Z/3\Z.$ However,  for primes $p \ge 5,$ and some matroids $M$ representable over $\Z/p\Z,$ there are many nonequivalent realizations of $M.$ 

\begin{thm} \label{Z4 uniqueness}
If $\H$ and $\H'$ are realizations of a polymatroid $P$ over $\Z/4\Z,$ then $\H$ and $\H'$ are equivalent.
\end{thm}

\noindent A proof of Theorem \ref{Z4 uniqueness} can be based on the presentation in \cite{Ox} of the unique representability of regular matroids.

\begin{thm} \label{Z6 uniqueness}
If $\H$ and $\H'$ are realizations of a polymatroid $P$ over $\Z/6\Z,$ then $\H$ and $\H'$ are equivalent.
\end{thm}

The main idea behind our proofs of Theorems \ref{matroid Z2n}, \ref{matroid cyclic} and \ref{Z6 uniqueness} is to use the fact that unlike fields, groups can have endomorphisms which are neither trivial nor automorphisms.  Here is an example where this idea is applied to polymatroid representability.

Suppose $\Gamma$ is the direct sum of two groups $\Gamma_1$ and $\Gamma_2$ whose orders are relatively prime, and $P$ is a $\Gamma$-\possible polymatroid.  Then for any $S \subseteq E,~|\Gamma^{r(S)}|$ can be written uniquely as a product $a(S)b(S),$ where $a(S)$ divides $|\Gamma_1|^{|S|}$ and $b(S)$ divides $|\Gamma_2|^{|S|}.$  Let $r_1$ and $r_2$ be the set functions defined by $r_1(S) = \log_{|\Gamma_1|} a(S)$ and $r_2(S) = \log_{|\Gamma_2|} b(S).$
Even if $P$ is $\Gamma$-\possible $(E,r_1)$ and/or $(E, r_2)$  may not be polymatroids.

\begin{thm}
With the above assumptions and notation, $P$ is representable over $\Gamma$ if and only if $(E,r_1)$ is a representable polymatroid over $\Gamma_1$ and $(E,r_2)$ is a representable polymatroid  over $\Gamma_2. $ 
\end{thm}

\section{Critical theorem for finite groups} \label{Crapo-Rota}

One of the cornerstones of enumerative   matroid theory is  Crapo and Rota's critical theorem.  The characteristic polynomial of a matroid occurs in applications with astonishing frequency and many times the reason can be traced back to the ideas behind this result.  Here we establish a finite group analog of the critical theorem.  When $\G$ is the additive group of a vector space and $\H$ is the additive group of a subspace of that vector space, Theorem \ref{group Crapo-Rota} below is equivalent to Crapo and Rota's original result.   

A forerunner of Theorem \ref{group Crapo-Rota} and its dual, Theorem \ref{Dual Crapo-Rota} below,   for {\em abelian} groups is in Section 6 of a paper by Kung, Murty and Rota \cite{KMR}.   Indeed, with enough translation from the R\'edei functions in \cite{KMR} to the rank functions here, one can view Theorem \ref{group Crapo-Rota} applied to abelian groups as a special case of \cite[Theorem 10]{KMR}.

The classical Crapo-Rota critical theorem is usually given in a coordinate-free form.   Let $E=\{v_1, \dots, v_n\}$ be vectors in a finite vector space $V$ over a (finite) field $\Field,~W$ the span of $E,$ and $s = \dim_\Field W.$   Then the Crapo-Rota critical theorem says that the number of $k$-tuples $(l_1, \dots, l_k) \in (V^\ast)^k$ such that $E \cap \ker l_1 \cap \cdots \cap \ker l_k  = \emptyset$ is $|\Field|^{k(n-s)} \mathlarger{\chi}_{M[W]}(|\Field|^k).$  It is a linear algebra exercise to recognize that this is equivalent to Theorem \ref{field Crapo-Rota}  which  
relates the zero sets of row vectors of a matrix $A$ whose coefficients are in a finite field $\Field,$ to the characteristic polynomial of the matroid $M[A].$

Let $A$ be an $m \times n$ matrix with row space $R(A).$ For $w = (w_1, \dots, w_n) \in R(A)$ define the zero set of $w$ to be $Z(w) = \{i:w_i =  0\}.$

\begin{thm} \label{field Crapo-Rota}  \cite{CR}
Let $A$ be an $m \times n$ matrix with coefficients in a finite field $\Field.$   Then 
\begin{equation} \label{Crapo-Rota field formula}
\mathlarger{\chi}_{M[A]}(|\Field|^k) = |\{(u_1, \dots, u_k) \in R(A)^k: Z(u_1) \cap \cdots \cap Z(u_k) = \emptyset\}|.
\end{equation}
\end{thm}

For the remainder of this section we no longer assume that all the $\Gamma_x$ are isomorphic.  We will continue to suppress $b$ from the notation with the understanding that $b$ can be any fixed  real number greater than one.   To state the analog of  Theorem \ref{field Crapo-Rota} for groups and polymatroids we substitute $\H \le \G$ for the row space and  define an analog of $Z(w)$.   For $h \in \H$ let $I(h) = \{x \in E: h_x = 1_{\Gamma_x}.\}$  So,  if $A$ is a  matrix with coefficients in $\Field,~\H = R(A),$ and for all $x \in E,~\Gamma_x$ is the additive group of $\Field,$ then $I(h) = Z(h).$  
\begin{thm} \label{group Crapo-Rota} (Crapo-Rota for finite groups)
  Suppose $\H$ is a subgroup of $\G$ which realizes the polymatroid $P.$ Then 
  
 \begin{equation} \label{Crapo-Rota formula}
  \mathlarger{\chi}_P(b^k)=|\{(h_1, \dots, h_k) \in \H^k:  \displaystyle\bigcap^k_{j=1} I(h_j) = \emptyset\}|.
 \end{equation}
\end{thm}

\noindent Observe that if $\Gamma_x$ is the additive group of a fixed finite field $\Field$ for every $x,$ and $b=|\Field|,$ then Theorem 
\ref{group Crapo-Rota} becomes Theorem \ref{field Crapo-Rota}.

Our proof of this theorem closely follows the usual M\"obius inversion proof of the Crapo-Rota critical theorem.   To prepare for this, we start by demonstrating several ways that $I(h)$ for realized polymatroids is similar to $Z(w)$ for realized matroids.   The next corollary and two propositions are polymatroid analogs of standard properties of realized matroids.

\begin{cor} \label{I=flat}
If $h \in \H,$ then $S=I(h)$ is a flat of $P(\H).$
\end{cor}

\begin{proof}
If $x \notin S,$ then $h$ is a nontrivial element of  $\ker \pr_S: \H_{S \cup \{x\}} \to \H_S.$ Hence (\ref{one more}) shows that $r_\H(S \cup \{x\}) - r_\H(S)>0.$
\end{proof}

\begin{prop} \label{realizing coatoms}
If $S$ is a coatom of $\L_{P(\H)},$ then there exists $h \in \H$ such that $I(h) = S.$
\end{prop}

\begin{proof}
Let $S$ be a coatom of $\L_{P(\H)}$ and set $\H'= \ker \pr_S:\H \to \H_S.$  Since $\pr_{E-S} \H'$ is a realization of $P(\H)/S,~\H'$ is not the trivial group.  Now, any $h' \in \H'$ which is not the identity must satisfy $I(h')= S.$
\end{proof}

\begin{prop} \label{realizing flats}
If $S$ is a flat of $P$ then there exists $\{h_1, \dots, h_l\},~l \le |E-S|,$ such that $S = \displaystyle\cap^l_{i=1} I(h_i).$
\end{prop}

\begin{proof}
Let $x \in E-S.$ Since $S$ is a flat, $x$ is not a loop of $P(\H)/S.$  Therefore, there exists $h \in \ker \pr_S: \H \to \H_S$ such that $h_x \neq 1_\Gamma.$ For each $x \in E-S$ choose such an $h$ and denote it by $h(x).$ Then 
$$S = \bigcap_{x \in E-S} I(h(x))$$
as required. 
\end{proof}

 \begin{proof} (of Theorem \ref{group Crapo-Rota})
The proof follows the same pattern as the usual M\"obius inversion proof of the Crapo-Rota critical theorem for finite fields.   If $x$ is a loop of $P$,  then for all $h \in \H,~ h_x = 1_{\Gamma_x}$  and both sides of (\ref{Crapo-Rota formula}) are zero.  

So assume that $P$ has no loops. Thus the empty set is a flat of $\L_{P(\H)}$ and we can apply Proposition \ref{mobius char poly} when needed later.  Since $I(h)$ is always a flat of $P$, and the flats of $P$ are closed under intersection, we can define $\phi:\H^k \to \L_{P(\H)}$ by $\phi(h_1, \dots, h_k) = I(h_1) \cap \cdots \cap I(h_k).$  Now let $f:\L_{P(\H)} \to \Z_{\ge 0}$ be defined by $f(S) = |\phi^{-1}(S)|$.   Then $\displaystyle\sum_{T \supseteq S} f(T)$ is the number of $(h_1, \dots, h_k)  \in \H^k$ with $(h_i)_x = 1_{\Gamma_x}$ for all $1 \le i \le k$ and $x \in S.$ This is 
$$|\ker \pr_S: \H \to \H_S|^k =\bigg( \frac{|\H|}{|\H_S|} \bigg)^k =( b^{r(E) - r(S)})^k=(b^k)^{r(E) - r(S)}.$$
 Lemma \ref{mobius char poly} and M\"obius inversion finish the proof. 

\end{proof}

In \cite {H} Helgason used M\"obius inversion to prove that the classical relationship between the chromatic polynomial of a graph and the characteristic polynomial of the cycle matroid of the graph extends to hypergraph coloring and the characteristic polynomial of a related polymatroid. Later,  Whittle provided a deletion/contraction proof of the same result \cite{Whitt}.   Whittle also shows how to relate this to a polymatroid version of the Crapo-Rota critical theorem for polymatroids representable over a field whenever the number of colors is a prime power.  In addition, he discusses how to  extend these ideas to an arbitrary number of colors  by using Dowling geometries.  Here we use Theorem \ref{group Crapo-Rota} to relate hypergraph coloring to critical problems over groups for any number of colors.  In section \ref{Polymatroid duality} we will use Theorem \ref{group Crapo-Rota} to establish a hypergraph  analog of the usual duality between nowhere-zero flows and proper colorings for graphs.

A {\bf hypergraph} is a pair $H=(E,V)$ consisting of a finite vertex set $V$ and a finite multiset of hyperedges $E.$ Each hyperedge is a finite multiset of $V.$  When necessary we write $E(H)$ and $V(H)$ for the hyperedges and vertices of $H$. We will use $|x|$ to indicate the cardinality of a hyperedge $x$ as a multiset.  So, if $x = \{a,a,b,b,b,c\},$ then $|x|=6.$ Given an undirected graph $G$ we can associate a cryptomorphic hypergraph $H(G)$ whose vertices are the vertices of the graphs, loops based at a vertex $\{a\}$ are listed as a hyperedge $\{a,a\},$ and edges with distinct incident vertices $a$ and $b$ are listed as hyperedges $\{a,b\}$ as many times as there are edges in $G$ with end points $a$ and $b.$  

  We  associate two different graphs to a hypergraph $H.$   The first graph attached to $H$ is a union of star-like multigraphs and is denoted by $\SG(H).$ The vertices of $\SG(H)$ are equal to the vertices of $H.$ To enumerate the edges of $\SG(H)$ we first define a set of edges $E(x)$ for every $x \in E.$ Given a hyperedge $x$ of $H$ choose one vertex $v_x$ in $x.$      If $v \neq v_x$ has multiplicity $s$ in $x,$ then $E(x)$ has $s$ parallel edges whose end points are $v$ and $v_x.$   The set of edges in $E(x)$ also contains $t-1$ loops where $t$ is  the multiplicity of $v_x$ in $x.$ Hence, the total number of edges in $E(x)$ is $|x|-1.$ The set of edges of $\SG(H)$ is the union of all of the $E(x).$ Thus the total number of edges in $\SG(H)$ is $\displaystyle\sum_{x \in E} |x|-1.$  See Figure \ref{hypergraph} and Figure \ref{SG example} for an example of $\SG(H).$ If $H$ is equal to $H(G)$ for a graph $G,$ then $\SG(H)= G.$ 
  
  The second graph we associate to $H$ is the usual bipartite graph which contains the same information as  $H.$  We hold off its description until we need it in Section \ref{Polymatroid duality}.

For the remainder of this section $\Gamma$ is a finite {\em abelian} group.  A {\bf proper  $\Gamma$-coloring} of a hypergraph $H$ is a map $\phi:V \to \Gamma$ so that for all $x \in E$ there exist vertices $a, b$ in $x$ such that $\phi(a) \neq \phi(b).$ Observe that for a graph $G,$  proper $\Gamma$-colorings of $H(G)$ are the same as proper $\Gamma$-colorings of $G.$  The cardinality of $\Gamma$ determines the number of proper $\Gamma$-colorings of $H,$ so we let $C_H(\lambda)$ be the number of proper $\Gamma$-colorings of $H$ for any  $\Gamma$ of cardinality $\lambda.$  While  the group structure of $\Gamma$ is not needed to define proper colorings, it will be needed for the duality with nowhere-zero flows in section \ref{Polymatroid duality}. 

The polymatroid defined in \cite{Whitt} is isomorphic to the polymatroid $P(H)= (E, r_H)$ constructed from $M[\SG(H)]$ and the partition $\{E(x_1), \dots, E(x_n)\}$ of $E(\SG(H)).$   We analyze $P(H)$ using an adjacency matrix $A(H)$ of $\SG(H).$ 
   Start by ordering the edges and vertices of $H$,  $E =\{x_1, \dots, x_n\},~V=\{v_1, \dots, v_m\}.$  The rows of $A(H)$ are indexed by $V.$ The columns are labelled by  edges of $\SG(H)$  as follows.   The edges are listed so that those in $E(x_1)$ come first, those in $E(x_2)$ come second, up to $E(x_n).$ Then, for $S \subseteq E,~r_H(S)$ is the dimension of the column space spanned by the union of all of the $E(x),~x \in S.$  If we direct all the edges in $E(x)$ so that the tail is $v_x,$ then $A(H)$ is the usual adjacency matrix of a directed graph.  Specifically, the entry of $A(H)$ in a row labelled by $v \in V$ and column labelled by $\{v_x,w\} \in E(x)$ is given by

$$ \begin{cases} 1 & \mbox{ if } v=w \mbox{ and},~w \neq v_x.\\ 0 & v \notin x\mbox{ or } v=w \mbox{ and } w=v_x \\  -1 & \mbox{ if }v = v_x \mbox{ and } w \neq v_x. \end{cases}$$

\begin{example} \label{the example}
Let $V=\{a,b,c,d\}$ and $ E= \{\{a,b,c\}, \{a,b,b,d\}, \{a,b,b,b\},\{c,d\}\}$ and choose $v_x$ to be  the alphabetically last vertex in each hyperedge.  Then 
\begin{equation} \label{hypergraph adjacency matrix}
A(H) = 
\begin{bmatrix}
\ & \{c,a\} & \{c,b\} & \{d,a\} &\{d, b\} & \{d,b'\} & \{b'',a\} & \{b'',b\} &\{b'', b'\}&\{d,c\} \\
a & 1 & 0 & 1 & 0 & 0 & 1 & 0 & 0&0\\
b & 0 & 1 & 0 & 1 & 1 & -1 & 0 & 0& 0 \\
c & -1 & -1 & 0 & 0 & 0 & 0 & 0 & 0& 1\\
d & 0 & 0 & -1 & -1 & -1 & 0 & 0 & 0 & -1
\end{bmatrix}.
\end{equation}
\noindent The primes on some of the column labels is to allow us to distinguish distinct copies of the same vertex in a hyperedge. 
\end{example}

\begin{figure}[h]
\begin{center}
\begin{tikzpicture}

\node[main node, fill=black, label={north: $c$}] [minimum size=0.2cm] (c) at (0,0) {};
\node[main node, fill=black, label={north: $d$}] [minimum size=0.2cm] (d) at (5,0) {};
\node[main node, fill=black, label={south: $b$}] [minimum size=0.2cm] (b) at (4.5,2) {};
\node[main node, fill=black, label={south: $b'$}] [minimum size=0.2cm] (b') at (5,2) {};
\node[main node, fill=black, label={south: $b''$}] [minimum size=0.2cm] (b'') at (5.5,2) {};
\node[main node, fill=black, label={south: $a$}] [minimum size=0.2cm] (a) at (0,2) {};

\node [draw, ultra thick, color=cbgold, rounded corners=10pt,inner sep=14pt, fit=(c) (d), label={south: $x_4$}] (cd){};
\node [fit=(a) (b) (c), label={west: $x_1$}] (abc){};
\draw [ultra thick, color=cbblue, solid,rounded corners=10pt, inner sep=10pt] ($(abc.south west)+(0,-0.25)$) -- ($(abc.north west)+(0,0)$) -- ($(abc.north east)+(0,0)$)--($(abc.north east)+(0,-1)$)-- cycle;
\node [fit=(a) (b') (d), label={east: $x_2$}] (abd){};
\draw [ultra thick, color=cbgreen, solid,rounded corners=10pt, inner sep=10pt] ($(abc.south east)+(.5,0)$) -- ($(abc.north west)+(-.15,-1)$)-- ($(abc.north west)+(-.15,.15)$) --($(abc.north east)+(.5,.15)$)-- cycle;
\node [draw, ultra thick, color=cbpink, rounded corners=10pt,inner sep=14pt, fit=(a) (b''), label={north: $x_3$}] (ab){};
\end{tikzpicture}
\caption{$H$ from Example \ref{the example} - $b,b'$ and $b''$ represent the same vertex.} \label{hypergraph}
\end{center}
\end{figure}
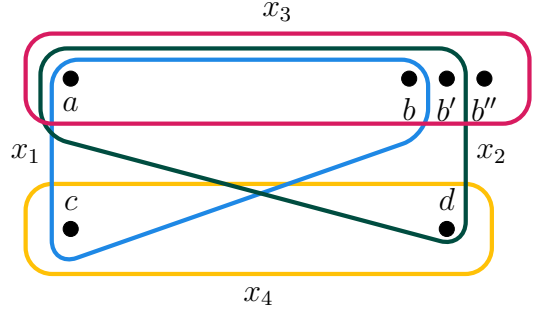

\begin{figure}[h]
\centering
  \begin{tikzpicture}
    \node[main node, fill=black, label={west: $a$}] [minimum size=0.2cm] (1) {} ;
    \node[main node, fill=black, label={east: $b$}] [minimum size=.2cm] (2) [right = 5cm of 1] {};
    \node[main node, fill=black, label={west: $c$}] [minimum size=.2cm] (3) [below = 4cm of 1] {};
     \node[main node, fill=black, label={east: $d$}] [minimum size=.2cm] (4) [below = 4cm of 2] {};
        
    \path[draw, ultra thick, cbpink]
    (2) edge node {} (1);
    \path[-,every loop/.style={out=145, in=35, looseness=70}, draw, ultra thick, cbpink]
    (2) edge [loop above] node {} ();  
    \path[-,every loop/.style={out=130, in=50, looseness=50}, draw, ultra thick, cbpink]
    (2) edge [loop above] node {} ();      
    \path[draw, ultra thick, cbblue]
    (1) edge node {} (3)
    (2) edge node {} (3);
    \path[draw, ultra thick, out=240, in=120, cbgreen]
    (2) edge node {} (4);
    \path[draw, ultra thick, out=300, in=60, cbgreen]
    (2) edge node {} (4);
    \path[draw, ultra thick, cbgreen]
    (1) edge node {} (4);
    \path[draw, ultra thick, cbgold]
    (4) edge node {} (3);
    
  \end{tikzpicture}
\caption{$\SG(H)$ - {\textcolor{cbblue}{$E(x_1)$}},  {\textcolor{cbgreen}{$E(x_2)$}},  {\textcolor{cbpink}{$E(x_3)$}},  {\textcolor{cbgold}{$E(x_4)$}}}  \label{SG example}
\end{figure}
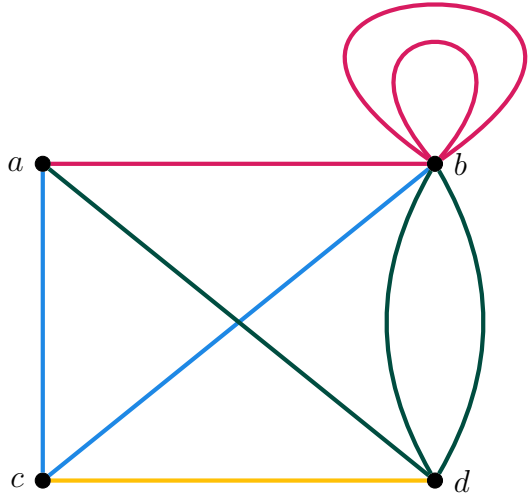

In \cite{Whitt} Whittle attributes the following theorem to Helgason \cite{H} and gives a deletion/contraction argument.  We will give a critical theorem proof.  When $H=H(G)$ for a graph $G,$ Theorem \ref{hypergraph coloring} is the usual relationship between the chromatic polynomial of a graph and the characteristic polynomial of its cycle matroid.   Let $\kappa(H)$ be the number of components of $\SG(H).$

\begin{thm}   \label{hypergraph coloring}
\begin{equation}
C_H(\lambda) = \lambda^{\kappa(H)} \mathlarger{\chi}_{P(H)}(\lambda).
\end{equation}
\end{thm}

\begin{proof}

Let $A=A(H).$ So $A$ is an $m \times k$ matrix with $k = \sum^n_{i=1} |x_i|-1.$ View elements $\gamma = (\gamma_1, \dots, \gamma_m) \in \Gamma^m$ as row vectors. Since $\Gamma^m$ is a $\Z$-module, it makes sense to set $\H$ to  be the image of $\Gamma^m$ under right-multiplication by $A$  in $\G=\Gamma^k.$  Then, setting $b=|\Gamma|,~P(\H)=P(\H,b)$ is a polymatroid with ground set  $[k]$ once we identify $[k]$ with the columns of $A.$

Claim: $P(\H)$ is isomorphic to $M[A].$ To see this, consider $T$ a subset of columns of $A$ and $D$ a basis of $T$ in the matroid $M[A].$    Since $A$ is the adjacency matrix of a graph, $A$ is totally unimodular.   Hence,  there is a row echelon form $A'$ of $A,$ obtained by integral row operations, such that the pivot columns of $A'$ contained in $T$ are the columns in $D.$   This implies that $\H_T$ is isomorphic to $\Gamma^{r_A(T)}$ and $r_\H(T) =r_A(T).$  

As noted above, $P(H)$ is the polymatroid constructed from $M[A]$ and the partition $\{E(x_1), \dots, E(x_n)\}$ of the columns of $A.$ Let $\G'=\displaystyle\prod^n_{i=1} \Gamma^{|x_i|-1}$ and  $f:\Gamma^k \to \displaystyle\prod^n_{i=1} \Gamma^{|x_i|-1}$ be the  isomorphism which maps the first $|x_1|-1$ coordinates of $\Gamma^k$ to $\Gamma^{|x_1|-1}$, the next $|x_2|-1$ coordinates to $\Gamma^{|x_2|-1},$ etc. This gives us a homomorphism $\psi: \Gamma^m \to \G'$ defined by $\psi(\gamma) = f(\gamma \cdot A).$ Apply $\psi$ to $\Gamma^m$ to get a subgroup $\H'$ of $\G'.$  Then $P(\H')$ is the polymatroid constructed from $P(\H)$ and the partition $\{E(x_1), \dots, E(x_n)\}$ of $[k],$ and hence is isomorphic to $P(H).$ 

To an element $\gamma=(\gamma_1, \dots, \gamma_m) \in \Gamma^m$ we associate a $\Gamma$-coloring $\phi_\gamma$ of $H$ by $\phi_\gamma(v_i) = \gamma_i.$ 
The definition of $A$ shows that $\phi_\gamma$ is a proper $\Gamma$-coloring of $H$ if and only if $I(\psi(\gamma)) = \emptyset.$  By Theorem \ref{group Crapo-Rota} there are $\mathlarger{\chi}_{P(\H)}(b)$ distinct images  $\psi(\gamma)$ such that $I(\psi(\gamma)) = \emptyset.$   Every $\psi(\gamma)$ in $\H'$ has $|\ker \psi|$ preimages.  Examining $A$ we see that $\gamma$ is in the kernel of $\psi$ if and only of $\gamma_i=\gamma_j$ whenever $v_i$ and $v_j$ are in the same component of $\SG(H).$ Thus $|\ker \psi|=b^{\kappa(H)}$ and there are $b^{\kappa(H)} \mathlarger{\chi}_{P(H)}(b) = |\Gamma|^{\kappa(H)} \mathlarger{\chi}_{P(H)}(| \Gamma |)$ proper $\Gamma$-colorings of $H.$

\end{proof}

\section{Polymatroid duality} \label{Polymatroid duality}
 In contrast to matroid duality, there are several different notions of polymatroid duality in the literature. See, for instance, \cite{CMSW},\cite{EL}, \cite{JMW}.   One of the most general appears to be $\a$-duality.  Let $\a \in \R^E_{\ge 0}.$  Then a polymatroid $P$ is an  $\a$-polymatroid if for all $x \in E,~r(\{x\}) \le \a_x.$   We will use $\a[k]$ to represent $\a$ such that $\a_x=k$ for all $x \in E.$  Hence, an integer polymatroid is a $k$-polymatroid if and only if it is an $\a[k]$-polymatroid.   Similarly, $P$ is subcardinal if and only if it is an $\a[1]$-polymatroid.
  
Given $\a \in (\R_{\ge 0})^E$ and $S \subseteq E$ we define $|\a_S| = \displaystyle\sum_{x \in S} \a_x.$ We will use $|\a|$ for $|\a_E|.$   Let $P$ be an $\a$-polymatroid.  Define $r^{\ast_\a}: 2^E \to \R_{\ge 0}$ by
  \begin{equation} \label{a-duality}
  r^{\ast_\a}(S) = r(E-S) + |\a_S| - r(E).
  \end{equation} 
  
  \noindent We omit the routine proof of the following.  
  
    \begin{prop} \label{a-duality prop}
  Let $P$ be an $\a$-polymatroid.  Then $P^{\ast_a} = (E, r^{\ast_\a})$ is a polymatroid. Furthermore, 
  \begin{enumerate}
  
    \item $r^{\ast_\a}(E) = |\a| - r(E).$
  \item $(P^{\ast_\a})^{\ast_\a} = P.$
  \item If $S \subseteq E,$ then $(P^{\ast_\a} \setminus S)^{\ast_\a} = P/S.$
  \item If $S \subseteq E,$ then $(P^{\ast_\a}/S)^{\ast_\a} = P \setminus S.$ 
  
 \end{enumerate}

    \end{prop}
  
\begin{defn}
Let $P=(E,r)$ be an $\a$-polymatroid.  The {\bf $\a$-dual} of $P$ is $P^{\ast_\a}= (E, r^{\ast_\a}),$ where $r^{\ast_\a}$ is defined by (\ref{a-duality}). 
\end{defn}
With respect to  $\a[1]$-duality, (\ref{a-duality})  coincides  with the usual rank function for matroid duality.   Observe that if $M$ is a  matroid, $P$ is a polymatroid constructed from a partition $\{S_1, \dots, S_n\}$ of $E(M),$ and $\a_i = |S_i|,$ then $P^{\ast_\a}$ is constructed from $M^\ast$ using the same partition.    When applied to $k$-polymatroids, $\a[k]$-duality  is usually called $k$-duality and has appeared before.   See, for instance, \cite{Whitt3}.  

 Until further notice, we assume that $\Gamma$ is an abelian group. For the moment, we do not assume $\Gamma$ is finite.   Let us recall the usual definitions  for nowhere-zero $\Gamma$-flows on an oriented graph.  A $\Gamma$-flow on an oriented graph $G$ is a function $\phi:E(G) \to \Gamma$ such that the inflow equals the outflow at each vertex.  A $\Gamma$-flow is nowhere-zero if for all $x \in E(G),~\phi(x) \neq 0.$   In order to define $\Gamma$-flows on a hypergraph $H,$ we use the usual bipartite graph associated to a hypergraph.  

Given a hypergraph $H$ we form a bipartite graph $\BG(H)$ whose bipartition consists of the vertices of $H$ in one block and the hyperedges of $H$ in the other block.  The number of edges between vertex $v$ and hyperedge $x$ is equal to the multiplicity of $v$ in $x.$  All edges are oriented toward the vertices.  See Figure \ref{BG} for $\BG(H)$ when $H$ is the hypergraph from Example \ref{the example}.  A {\bf $\Gamma$-flow on $H$} is a $\Gamma$-flow on $\BG(H).$  A {\bf nowhere-zero $\Gamma$-flow on $H$} is a $\Gamma$-flow on $H$ such that for all hyperedges $x$ there exists at least one edge in $\BG(H)$ whose  head is $x$ and whose value is not zero.    
When $G$ is an oriented graph,  there is a bijection between  nowhere-zero $\Gamma$-flows on the graph $G$ and nowhere-zero $\Gamma$-flows on the hypergraph $H(G).$   Given a nowhere-zero $\Gamma$-flow $\phi:E(G) \to \Gamma$  on the graph $G,$ and an oriented edge $x=(v,w)$ of $G,$ define $\psi:E(\BG(H(G))) \to \Gamma$ by $\psi((v,x)) =- \phi(x)$ and $\psi((w,x)) = \phi(x).$  Then $\phi \leftrightarrow \psi$ is such a bijection.  Nowhere-zero $\Gamma$-flows on hypergraphs share another property with their graphic counterparts.

\begin{figure}[h]
\centering
  \begin{tikzpicture}
    \node[main node, fill=cbblue, label={west: $x_1$}] [minimum size=0.2cm] (1) {} ;
    \node[main node, fill=black, label={east: $a$}] [minimum size=.2cm] (2) [right = 5cm of 1] {};
    \node[main node, fill=cbgreen, label={west: $x_2$} ] [minimum size=.2cm] (3) [below = 3cm of 1] {};
    \node[main node, fill=black, label={east: $b$}] [minimum size=.2cm] (4) [below = 3cm of 2] {};
    \node[main node, fill=cbpink, label={west: $x_3$} ] [minimum size=.2cm] (5) [below = 3cm of 3] {};
    \node[main node, fill=black, label={east: $c$}] [minimum size=.2cm] (6) [below = 3cm of 4] {};
    \node[main node, fill=cbgold, label={west: $x_4$} ] [minimum size=.2cm] (7) [below = 3cm of 5] {};
    \node[main node, fill=black, label={east: $d$}] [minimum size=.2cm] (8) [below = 3cm of 6] {};

    \path[draw, ultra thick, cbblue, every node/.style={sloped,allow upside down}]
    (1) edge node {\tikz \draw[-triangle 90] (0,0) -- +(.1,0);} (2) 
    (1) edge node {\tikz \draw[-triangle 90] (0,0) -- +(3,0);} (4)    
    (1) edge node {\tikz \draw[-triangle 90] (0,0) -- +(6,0);} (6);
    \path[draw, ultra thick, cbgreen, every node/.style={sloped,allow upside down}]
    (3) edge node {\tikz \draw[-triangle 90] (0,0) -- +(2,0);} (2)
    (3) edge node {\tikz \draw[-triangle 90] (0,0) -- +(2,0);} (4)
    (3) edge node {\tikz \draw[-triangle 90] (0,0) -- +(5,0);} (8);
    \path[draw, ultra thick, out=20, in=165, cbgreen, every node/.style={sloped,allow upside down}]
    (3) edge node {\tikz \draw[-triangle 90] (0,0) -- +(.1,0);} (4);
    \path[draw, ultra thick, cbpink, every node/.style={sloped,allow upside down}]
    (5) edge node {\tikz \draw[-triangle 90] (0,0) -- +(5,0);} (2)
    (5) edge node {\tikz \draw[-triangle 90] (0,0) -- +(3,0);} (4);
    \path[draw, ultra thick, in=195, cbpink, every node/.style={sloped,allow upside down}]
    (5) edge node {\tikz \draw[-triangle 90] (0,0) -- +(.1,0);} (4);
    \path[draw, ultra thick, out=25, in=235, cbpink, every node/.style={sloped,allow upside down}]
    (5) edge node {\tikz \draw[-triangle 90] (0,0) -- +(.1,0);} (4);    
    \path[draw, ultra thick, every node/.style={sloped,allow upside down}, cbgold]
    (7) edge node {\tikz \draw[-triangle 90] (0,0) -- +(4,0);} (6)
    (7) edge node {\tikz \draw[-triangle 90] (0,0) -- +(3,0);} (8);
    
  \end{tikzpicture}
  \caption{$\BG(H)$}  \label{BG}
\label{BG(H)}
\end{figure}

\begin{prop}
A hypergraph has a nowhere-zero $\Z/m\Z$-flow if and only if it has a nowhere-zero $\Z$-flow $\phi$  with the all values of $\phi$ in $\{-(m-1), \dots, m-1\}.$  
\end{prop}

\begin{proof}
Given a nowhere-zero $\Z/m\Z$-flow on $H,$ remove all edges of $\BG(H)$ assigned $0$  and apply the same fact for graphs to see that $H$ has a nowhere-zero $\Z$-flow all of whose values are in $\{-(m-1), \dots, m-1\}$.   See, for instance, \cite[Proposition 6.3.7]{BO}.   Conversely, given a nowhere-zero flow on $H$ with values in $\{-(m-1), \dots, m-1\},$ reduce modulo $m$ to get a nowhere-zero $\Z/m\Z$-flow on $H.$
\end{proof}

One way to visualize $\Gamma$-flows on $H$ is to picture the hyperedges as hubs and the vertices as customers. For a customer $v$ and hub $x,$  the value of $\phi(v, (v,x))$ is the value of whatever the customer is sending the hub. In a $\Gamma$-flow the total going through the hub and the total being sent by each customer to all of the hubs is zero.  A nowhere-zero $\Gamma$-flow has something nonzero going through every hub.  From this point of view, flows treat vertices and hyperedges in an identical fashion.  The nowhere-zero condition only applies to the hyperedges (hubs).  

\begin{thm} \label{hypergraph flows}
Let $H$ be a hypergraph and $\Gamma$ a finite abelian group.  Let $\a \in (\R_{\ge 0})^{E(H)}$ be defined by $\a_x = |x|-1.$  Then the number of nowhere-zero $\Gamma$-flows on $H$ is $\mathlarger{\chi}_{P(H)^{\ast_\a}}(|\Gamma|).$
\end{thm}

\noindent If $H=H(G),$ then $P(H)$ is the cycle matroid of $G$ and $\a=\a[1]$ induces  standard matroid duality.  Therefore, in this case, Theorem \ref{hypergraph flows} is the classical flow polynomial formula for nowhere-zero $\Gamma$-flows on $G$ \cite{BO}.\\  

\begin{proof}
We follow the notation of the proof of Theorem \ref{hypergraph coloring}.  So $\G = \Gamma^k$ and $\G' = \displaystyle\prod^n_{i=1} \Gamma^{|x_i|-1}.$ This time we think of elements $\gamma = (\gamma_1, \dots, \gamma_k)$ of $\G$ as column vectors and define $\tau: \Gamma^k \to \Gamma^m$   by $\tau(\gamma) = A \cdot \gamma.$   Now we let $\KK' = f(\ker \tau) \subseteq \G',$ where $f$ is  the same isomorphism between $\Gamma^k$ and $\G'$ as before.  By Theorem \ref{group Crapo-Rota}, the number of $h \in \KK'$ such that $I(h) = \emptyset$ is $\mathlarger{\chi}_{P(\KK')}(b)= \mathlarger{\chi}_{P(\KK')}(|\Gamma|).$  To finish the proof we establish the following three claims:
\begin{enumerate}
\item There is a bijection $\alpha$ from $\KK'$ to $\Gamma$-flows on $H.$ 
\item  When restricted to $\FF = \{h \in \KK': I(h) = \emptyset\}, ~\alpha$ is a bijection between $\FF$ and nowhere-zero $\Gamma$-flows on $H.$ 
\item $P(\KK') = P(H)^{\ast_\a}.$
\end{enumerate}
Let $\gamma=((\gamma_{1,1}, \dots, \gamma_{1,|x_1|-1}), \dots, (\gamma_{n,1}, \dots, \gamma_{n, |x_n|-1})) \in \G'.$  For $x_j \in E$ write down the vertices in $x_j= (w_1, \dots, w_{|x_j|-1}, w_{|x_j|})$ in order so that $(\{v_{x_j}, w_1\}, \{v_{x_j}, w_2\}, \dots, \{v_{x_j}, w_{|x_j|-1}\})$ is the same ordering of the edges in $E(x_j)$ as was used to define $A$ and $w_{|x_j|} = v_{x_j}.$   We define a function $\alpha(\gamma):E(\BG(H)) \to \Gamma$  as follows.  
       \begin{equation} \label{F} \alpha(\gamma)(w_i, x_j) = \begin{cases} \gamma_{j,i }~,& i < |x_j| \\ - \displaystyle\sum^{|x_j|-1}_{l=1} \gamma_{j,l}~, & i=|x_j| \end{cases}.
       \end{equation}
 The definition of $\alpha$ implies that for $\alpha(\gamma)$ the sum of outgoing edges to a hyperedge $x$ in $\BG$ is zero.  The $\gamma \in \KK'$ are precisely those such that $\alpha(\gamma)$ is a $\Gamma$-flow on $H$ establishing the first claim.     The second claim follows from the definition of $I(h).$ It remains to prove that  $P(\KK') = P(H)^{\ast_\a}.$

First consider $M[A]$.  A result that goes back to Whitney \cite{Wh} is that the if $\KK$ is the kernel of multiplication by $A$ as a linear transformation of real vector spaces, then   $M[\KK] \cong M[A]^\ast.$  As in the proof of Theorem \ref{hypergraph coloring}, we identify the columns of $A$ with $[k]$ and let $T$ be a subset of the columns of $A.$ Then $r^\ast_A(T),$ the rank of $T$ in $M[A]^\ast,$ is the maximum number of free variables in $T$ over all all matrices row equivalent to $A.$   Let $A'$ be a  matrix row equivalent to   $A$ which maximizes  the number free variables in $T.$ We can choose $A'$  to be the identity matrix when restricted to the pivot columns. Since $A$ is totally unimodular, the rows of $A'$ are an integral linear combination of the rows of $A.$   Therefore, the elements of $\ker \tau$ consist of arbitrary choices in $\Gamma$ for each coordinate corresponding to a free variable, and all coordinates associated  with pivot variables are determined by those choices.  Hence, $|\pr_T(\ker \tau)|=|\Gamma|^{r^\ast_A(T)}$ and $P(\KK)$ equals $M[A]^\ast.$  Since $P(\KK')$ is constructed from $P(\KK)$ and the partition $\{E(x_1), \dots, E(x_n)\}$ and $P(H)^{\ast_\a}$ is constructed from $M[A]^\ast$ with the same partition, they are equal.

\end{proof}

Three special situations deserve mention.  One, if $G$ is a graph and $H=H(G)$, then Theorem \ref{hypergraph coloring} and  Theorem \ref{hypergraph flows}  are the usual coloring-flow duality for graphs.  Two, if $H$ is a uniform $k$-hypergraph, then $\a=\a[k]$ and $\ast_\a$ is $k$-polymatroid duality.  Lastly, if the edges of $H$ are subsets of $V$ (as opposed to multisets), then $|x|-1= r_{P(H)}(\{x\}).$  Basing polymatroid duality on the rank of the singletons has been considered before.  See, for instance, \cite{CMSW} and \cite{Whitt3}.

\section{Duality for polymatroids realized over a group}  \label{duality for realized polymatroids}
One of the  motivations for Whitney's original introduction of matroids \cite{Wh} was to answer the question, ``What is the dual of a nonplanar graph?''    He offered two solutions.  One, let $M[G]$ be the cycle matroid of the graph $G$.  Then $M[G]^\ast$ serves as the dual of $G.$ The primary justification for this point of view is that for embedded planar graphs, $M[G]^\ast$ is isomorphic to $M[G^\ast],$ where $G^\ast$ is the geometric planar dual of $G.$ 

In order to describe Whitney's second answer, choose an orientation for $G$ and let $A$ be the adjacency matrix for $G,$ the columns labelled by the edges and the rows by the vertices.  Now $M[A]$ is isomorphic to $M[G].$ Then $R(A)^\perp,$ the orthogonal complement of the row space of $A,$ serves as a dual for $G.$ This  point of view is validated by the fact that  $M[R(A)^\perp]$ is isomorphic to $M[A]^\ast$ which itself is isomorphic to  $M[G]^\ast.$

In this section we will closely mirror these ideas in the search for a dual of $\H \le \G.$ Throughout this section $\G = \displaystyle\prod_{x \in E} \Gamma_x,$ where each $\Gamma_x$ is a finite group (not necessarily equal), and $\H$ is a subgroup of $\G.$  As before, we fix $b >1.$  In addition, we set  $\a \in (\R_{\ge 0})^E,~\a_x = \log_b |\Gamma_x|.$ We will be comparing our constructions with the special situation when for all $x \in E,~\Gamma_x = \Z/p\Z,~p$ a fixed prime. In this case, $P(\H)$ is the matroid previously denoted by $M[\H],$ where $\H$ is a subspace of the vector space $(\Z/p\Z)^E.$

We have already seen the analog of Whitney's first solution.   The polymatroid $P(\H,b)^{\ast_\a}$ is a polymatroid dual of $P(\H,b).$ When $\Gamma_x = \Z/p\Z$ for all $x,~P(\H,b)^{\ast_\a} = M[\H]^\ast,$ the usual matroid dual of $M[\H].$  Our analog of Whitney's second solution to ``What is the dual of a nonplanar graph'' involves the (complex) representation theory of finite groups.  

We will use the following notation to describe a candidate for a representation-based dual to $\H \le \G.$  For a finite set $U,~\C[U]$ is the $\C$-vector space with  basis $\{w_u: u \in U\}.$  When $\Gamma$ is a finite group, $\C[\Gamma]$ will also stand for the  left-regular representation of $\Gamma.$ Specifically, if $\displaystyle\sum_{\gamma \in \Gamma} c_\gamma~w_\gamma \in \C[\Gamma]$ and $\gamma' \in \Gamma,$ then $\gamma' \cdot \displaystyle\sum_{\gamma \in \Gamma} c_\gamma~w_\gamma = \displaystyle\sum_{\gamma \in \Gamma} c_\gamma~w_{\gamma' \cdot \gamma}.$   The set of (complex) irreducible representations of $\Gamma$ is $\widehat{\Gamma.}$  The irreducible representations of $\widehat{\G}$ are the representations of the form $\rho = \displaystyle\bigotimes_{x \in E} \rho_x,~\rho_x \in \widehat{\Gamma_x}$ \cite{Se}.

Depending on the circumstances, the representation-based dual of $\H \le \G$ is one of the following three cryptomorphic objects.

\begin{defn} \ \\ 
  \begin{itemize}
     \item $\Rep_1(\H)$: The subrepresentation  of $\C[\G]$ consisting of all elements whose coefficients are constant on the right cosets of $\H.$ Equivalently, $$\Rep_1(\H) = \bigg\{ \displaystyle\sum_{g \in \G} c_g w_g \in \C[\G]: \mbox{ for all } g \in \G,~h \in \H, c_g = c_{g \cdot h}\bigg\}.$$
     \item $\Rep_2(\H)$: The permutation representation on $\C[\G/\H]$ induced from the $\G$-action on right-cosets $\G/\H,~ g' \cdot (g \H) = (g' \cdot g) \H.$
     \item $\Rep(\H)$: The multiset of irreducible representations in $\widehat{\G}$ whose direct sum is equivalent to $\Rep_1(\H)$ and $\Rep_2(\H).$ 
  \end{itemize}
\end{defn}

\noindent In representation theory $\Rep_1(\H)$ and  $\Rep_2(\H)$ are known as the  representation of $\G$ induced by the  trivial representation of $\H.$ 

\begin{example} \label{diagonal}
 Let $\G = S_3 \times S_3, \H = \{(\sigma, \sigma): \sigma \in S_3\}, $ the diagonal embedding of $S_3.$   A character computation applied to the permutation representation $\Rep_2(\H)$ shows that $\Rep(\H) = \{1 \otimes 1,   s \otimes s, t \otimes t\},$ where $1, s$ and $t$ are respectively the trivial, sign and two-dimensional irreducible representations of $S_3.$ 
\end{example}

\begin{prop} \label{normal subgroups}
 Let $\H$ be a normal subgroup of $\G$ and $\pr:\G \to \G/\H$ be the quotient map.  Mapping $\phi \in \widehat{\G/\H}$ to $\Rep(\H)$ by $\phi \to \phi \circ \pr$ is a bijection between the multisets $\widehat{\G/\H}$ and $\Rep(\H).$   Equivalently, $\rho \in \widehat{\G}$ is in $\Rep(H)$ if and only if $\H \subseteq \ker \rho$, and, if $\rho \in \Rep(\H), \rho = \phi \circ \pi,$ then the multiplicity of $\rho$ in $\Rep(\H)$ equals the multiplicity of $\phi$ in $\C[\G/\H].$
  \end{prop}
  
 \begin{proof}
 From the point of view of $\Rep_2(\H),$ we see that  if $\phi \in \widehat{\G/\H},$ then  $\phi \circ \pr$ is in $\Rep(\H).$ Since the $\dim \Rep_2(\H) = |\G/\H|$ and $\phi \to \phi \circ \pr$ as a map from $\widehat{\G/\H} \to \Rep(\H)$ is injective, $\phi \circ \pr$ accounts for all of $\Rep(\H).$  
 \end{proof} 
 
 \begin{example}
 $\G=S_3 \times S_3$ and  $\H = \{(\sigma, \tau): s(\sigma)=s(\tau)\},$ where as in Example \ref{diagonal}, $s$ is the sign representation of $S_3.$ Then $\H$ is contained in the kernel of $1 \otimes 1$ and equals the kernel of $s \otimes s.$  Since $|\G/\H| = 2, \Rep(\H) = \{1 \otimes 1, s \otimes s\}.$
 
 \end{example}
 
 When $\Gamma_x$ is a fixed abelian group for all $x \in E,$ the idea that the subgroup of representations in $\widehat{\G}$ whose kernels contain $\H$ could be used as a dual code to $\H$ has been part of the coding literature since \cite{Mc}.   In Section \ref{Group codes} we will briefly explore how $\Rep(\H)$ acts as a dual code for $\H$ whether or not $\H$ is abelian.  
 
 \begin{example}
 Suppose $\G = (\Z/p\Z)^E,~E=[n],$ and $p$ is a prime.  As in Example \ref{Z/pZ}, a subgroup $\H$ of $\G$ is also a subspace of $~\G$ viewed as a vector space.  How is $\Rep(\H)$ related to $\H^\perp?$ The irreducible representations of $\Z/p\Z$ are $\{\rho_0, \rho_1, \dots, \rho_{p-1}\}$ with $\rho_j(\bar{k}) = e^{2\pi i (jk/p)}.$  The function $f: \widehat{\Z/p\Z} \to \Z/p\Z,~f(\rho_j) = \bar{j}$ is an isomorphism of groups,  the group operation on $\widehat{\Z/p\Z}$ being tensor product.   Let $F: \widehat{\G} \to \G$ be the  isomorphism induced by $f.$ Suppose $\rho \in \Rep(\H).$ Write $F(\rho) = (\bar{m}_1, \dots, \bar{m}_n)$ and for $h \in \H$ write $h=(\bar{h}_1, \dots, \bar{h}_n).$ Then $\rho \in \Rep(\H)$ if and only if $\rho(h) =1$ for all $h \in \H.$ This holds if and only if $\displaystyle\prod^n_{k=1} e^{2 \pi i (m_k h_k/p)} = 1.$ Equivalently,   if and only if $\displaystyle\sum^n_{k=1} \bar{m}_k \bar{h}_k = 0 \mod p.$ This is  Whitney's description of the subspace $\H^\perp$ of $\G$ such that $M[\H]^\ast = M[\H^\perp].$ 
 \end{example}

  The rest of this section shows various ways $\Rep(\H)$ realizes $P(\H)^{\ast_\a}$ in a manner similar to the way $W^\perp$ realizes $M[W]^\ast$ when $W$ is a subspace of $\Field^E$ with $\Field$  a field. In general, the trivial representation in $\widehat{\G}$ will play the role of the $0 \in \Field,$ the tensor product decomposition of representations in $\widehat{\G}$ substitutes for the coordinates of elements of $\Field^E,$ and dimensions of representations will replace counting statistics.  For notational brevity we will denote $P^{\ast_a}(\H)$ by $P^\ast$ and the corresponding rank function as $r^\ast.$ First a few preparatory lemmas.  
  
   A {\bf coloop} of an $\a$-polymatroid is an element $x \in E$ such that $r(E) = r(E-\{x\}) + r(\{x\})$ and $r(\{x\}) = a_x > 0.$   When $\a= \a[1]$ this definition of coloop is the usual one for matroids (and subcardinal polymatroids).  The law of diminishing returns implies that if $x$ is a coloop of an $\a$-polymatroid and $x \notin S,$ then $r(S \cup \{x\}) = r(S) + r(\{x\}).$ We omit the routine proof of the following characterization of the flats of $P^\ast.$
  \begin{lem}
  If $S$ is a subset of an $\a$-polymatroid $P=(E,r),$ then $S$ is a flat of $P^\ast$ if and only if $P \setminus S$ contains no coloops. 
  \end{lem}
  
  \noindent For all $\a$-polymatroids $P=(E, r), ~r^\ast (E) = |\a|-r(E).$  
  
  \begin{lem} \label{r* for E}
  $$\log_b \displaystyle\sum_{\rho \in \Rep(\H)} \dim \rho = |\a|-r_\H(E).$$
  \end{lem}
  \begin{proof}
  By definition $\displaystyle\sum_{\rho \in \Rep(\H)} \dim \rho = \dim \Rep_2(\H) = |\G/\H|.$
  \end{proof}
  
  \begin{defn}
Let $\rho \in \widehat{\G},~\rho = \displaystyle\bigotimes_{x \in E} \rho_x,$ each $\rho_x \in \widehat{\Gamma}_x.$  Then, $$\triv(\rho) = \{x \in E: \rho_x \mbox{ is the trivial representation of } \Gamma_x\}.$$
\end{defn}  

\begin{example}
Suppose $\G= (\Z/p\Z)^n, \H$ is a subgroup of $\G,$ and $F: \Rep(\H) \to \G$ is the isomorphism from $\widehat{\G} \to \G$ restricted to $\Rep(\H).$ Then $\triv(h) = I(F(h)).$
\end{example}

\begin{example}  \label{s, S3}
Let $E=\{x,y\},~\G = \Gamma_x \times \Gamma_y,~\Gamma_x = \Gamma_y = S_3.$  Let 
$$\H=\{((1),(1)), ((1),(123)), ((1),(132)), ((12),(12)), ((12),(13)), ((12),(23))\}.$$
A character calculation shows that 
$$\Rep(\H) = \{1 \otimes 1, s \otimes s, t \otimes 1, t \otimes s\}.$$
 So $\triv(1 \otimes 1) = \{x,y\}, \triv( t \otimes 1) = \{y\}, \triv(s \otimes s) = \triv(t \otimes s) = \emptyset.$ 
\end{example}

We are going to establish $\Rep(\H)$-mirrors of Corollary \ref{I=flat} and Propositions \ref{realizing coatoms},  and \ref{realizing flats}.  The next two lemmas prepare us for those results.  Let $S \subseteq E, \rho \in \widehat{\G_S}, \rho= \displaystyle\bigotimes_{x \in S} \rho_x.$ The {\bf extension of $\rho$ to $\widehat{\G}$} is $\bar{\rho} = \displaystyle\bigotimes_{x \in E} \bar{\rho}_x,$ where  $\bar{\rho}_x = \rho_x$ for all $x \in S,$ and for $x \notin S,~\bar{\rho}_x \mbox{ is the trivial representation of } \Gamma_x.$

\begin{lem} (Extension lemma) 
Let $S \subseteq E.$ Then $\rho$ is in $\Rep(\H_S)$ if and only if $\bar{\rho}$ is in  $\Rep(\H).$
\end{lem}
\begin{proof}
The easiest way to see this is from the point of view of $\Rep_1(\H)$ and to observe that  as a subrepresentation of $\C[\G],~\rho$ is constant on the right cosets $\G_S/\H_S$ if and only if $\bar{\rho}$ is constant on the right cosets $\G/\H.$  
\end{proof}

We saw previously that $S$ is a flat of $P^\ast$ if and only if $E-S$ has no coloops.  Hence it might be useful to understand  the relationship between $\Rep(\H)$ and the coloops of $P(\H).$

\begin{lem} \label{dual coloops}
An element $x \in E$ is a coloop of $P(\H)$ if and only if $x \in \triv (\rho)$ for all $\rho \in \Rep(\H).$
\end{lem}  

\begin{proof}
  By definition, $x$ is a coloop of $P(\H)$ if and only if  
  $$|\H| = |\H_{E-\{x\}}| |\Gamma_x| \mbox{ iff } \frac{|\G|}{|\H|} = \frac{|\G_{E-\{x\}}|}{|\H_{E-\{x\}}|} \mbox{ iff } \dim \Rep_2(\H) = \dim \Rep_2(\H_{E - \{x\}}).$$

The extension lemma implies that the  last equality holds if and only $x \in \triv(\rho)$ for all $\rho \in \Rep(\H).$
\end{proof}

\begin{prop} \label{triv=flat}
If $\rho \in \Rep(\H),$ then $ \triv(\rho)$ is a flat of $P^\ast.$
\end{prop}

\begin{proof}
Let $S = \triv(\rho).$ If $x \in E - S$ is a coloop of $P(\H) \setminus S,$ then the previous proposition implies that $ x \in \triv(\rho). $   Therefore, $E-S$ has no coloops and $S$ is a flat of $P^\ast.$
\end{proof}

\begin{prop}  \label{dual coatoms}
If $S$ is a coatom of $\L_{P^{\ast}},$ then there exists $\rho \in \Rep(\H)$ suchthat $\triv(\rho) = S.$
\end{prop}

\begin{proof}
  Since $S$ is a coatom of $\L_{P^\ast}(\H),~\dim \Rep_2(\H_{E-S}) > 0,$ so there exists $\rho \in \Rep(\H_{E-S}),~ \rho$ not equal to the trivial representation.  If $\triv(\rho) \neq \emptyset,$ then $\triv(\bar{\rho})$ is a flat of $P^{\ast}$ strictly between $S$ and $E.$ Thus $\triv(\rho)=\emptyset, \triv(\bar{\rho})=S$ and $\bar{\rho} \in \Rep(\H).$
\end{proof}

\begin{prop}
 Let $S$ be a flat of $P^\ast.$ Then there exist $\{\rho_1, \dots, \rho_m\} \subseteq \Rep(\H),~m \le n-|S|$ such that  $\bigcap^m_{i=1} \triv(\rho_i) = S.$ 
\end{prop}

\begin{proof}
Let $y \in E-S.$  Since $y$ is not a coloop of $P \setminus S,$  Proposition \ref{dual coatoms} shows that there exists $\rho^y \in \Rep(\H_{E-S})$ such that $\rho^y = \bigotimes_{x \in E-S} \rho^y_x,~\rho^y_x \in \widehat{\G_x}$ and $\rho^y_y$ is not the trivial representation of $\G_y.$ Then$\{\overline{\rho^y}\}_{y \in E-S}$ satisfies the conclusion of the proposition. 
\end{proof}

\begin{thm} (Dual Crapo-Rota for groups) \label{Dual Crapo-Rota} \ \\ 
Given $S \subseteq E,$ let $Y(S) = \{(\rho_1, \dots, \rho_k) \in (\Rep(\H))^k: \displaystyle\bigcap^k_{i=1} \triv(\rho_i) = S \}$. For  any choice of $b>1$ and positive integer $k,$  
\begin{equation} \label{dual Crapo-Rota formula}
\displaystyle\sum_{(\rho_1, \dots, \rho_k) \in Y(\emptyset)} \displaystyle\prod^k_{i=1} \dim \rho_i= \mathlarger{\chi}_{P^\ast}(b^k).
\end{equation}
\end{thm}

\begin{proof}
With Proposition \ref{triv=flat} in hand we can copy the proof of Theorem \ref{group Crapo-Rota}.  If $x$ is a coloop of $P(\H),$ then by the proof of Proposition \ref{triv=flat}, the left hand side of (\ref{dual Crapo-Rota formula}) is zero.  In addition, $x$ is a loop of $P^\ast,$ so $\mathlarger{\chi}_{P^\ast}$ is the zero function.  

So we assume that $P(\H)$ has no coloops and thus the empty set is a flat of $\L_{P^\ast}.$ Since $\triv(\rho)$ is always a flat of $P^\ast$, and the flats of $P^\ast$ are closed under intersection, $\phi:(\Rep(\H))^k \to \L_{P^\ast}$ defined by $\phi(\rho_1, \dots, \rho_k) = \triv(\rho_1) \cap \cdots \cap \triv(\rho_k)$ is a function from $(\Rep(\H))^k$ to $\L_{P^\ast}.$ Now let $f:\L_{P^\ast} \to \Z_{\ge 0}$ be defined by $f(S) = \displaystyle\sum_{(\rho_1, \dots, \rho_k) \in Y(S)} \displaystyle\prod^k_{i=1} \dim \rho_i.$   When $Y(S)$ is empty,  $f(S)$ is defined to be zero.   Then 
$$\displaystyle\sum_{T \supseteq S} f(T)=\displaystyle\sum_{\stackrel{(\rho_1, \dots, \rho_k)} {\forall i,~ S \subseteq \triv(\rho_i)}} \displaystyle\prod^k_{i=1} \dim \rho_i$$ 
 $$=(\dim \Rep_2(\H_{E-S}))^k = \bigg(\frac{|\G_{E-S}|}{|\H_{E-S}|}\bigg)^k=(b^{\a_{E-S} - r_\H(E-S)})^k= (b^k)^{r^\ast(E) - r^\ast(S)}.$$
 The second equals sign follows from the extension lemma and the last equals sign holds for any $\a$-polymatroid.   
 Now apply Lemma \ref{mobius char poly} and M\"obius inversion.
\end{proof}

\begin{example} (Example \ref{s, S3} continued)
Let $\Gamma, \G,$ and $ \H$ be as in Example \ref{s, S3}. Recall that $1,s,$ and $t$ are respectively, the trivial, sign and two-dimensional irreducible representations of $S_3.$ So $b=|\Gamma|=6,$ and $\a=(1,1).$ Hence,

$$r_\H(\emptyset) = 0,~r_\H(\{1\}) = \log_6 2,~r_\H(\{2\}) = 1 = r_\H(\{1,2\}).$$

Therefore,
$$r^{\ast_\a}(\emptyset) = 0,~r^{\ast_\a}_\H(\{1\}) = r^{\ast_\a}_\H(\{1,2\}) =1,~r^{\ast_\a}_\H(\{2\}) = \log_6 2.$$

The flats and relevant M\"obius functions in $P^{\ast_\a}$ are
$$\emptyset \le \{2\} \le \{1,2\},~\mu(\emptyset, \emptyset)=1,~\mu(\emptyset, \{2\}) = -1, \mu(\emptyset, \{1,2\})=0.$$

Hence,
$$\mathlarger{\chi}_{P^{\ast_\a}} (\lambda) = \lambda - \lambda^{1-\log_6 2} = \lambda - \lambda^{\log_6 3}.$$

The intersection $\triv(\rho_1) \cap \triv(\rho_2) = \emptyset$ for all pairs $(\rho_1, \rho_2) \subseteq (\Rep(\H))^2$ except
$$(1 \otimes 1, 1 \otimes 1), (1 \otimes 1, t \otimes 1), (t \otimes 1, 1 \otimes 1), (t \otimes 1, t \otimes 1).$$

Thus, the l.h.s. of (\ref {dual Crapo-Rota formula}) is $36-9.$ On the other side, 
$$\mathlarger{\chi}_{P^{\ast_a}}(6^2) = 36 - 36^{\log_6 3} = 36- 6^{2 \log_6 3} = 36-9.$$
\end{example}

As the reader may have noticed, we have not yet given a rank function on $E$  that only depends on the combinatorics of $\Rep(\H),$ equals $r^\ast_\H$, and reduces to the usual definition of $M[\H]^\ast$ when $\Gamma_x$ is $\Z/p\Z$ for all $x.$ Lemma  \ref{r* for E} shows that $r_{\Rep(\H)}(E) = \log_b \displaystyle\sum_{\rho \in \Rep(\H)} \dim \rho$  gives the same answer as the usual definition of $r^\ast_{\H^\perp}(E)$ for the special case of all $\Gamma_x = \Z/p\Z$. However,  it is not at this point clear how to define a rank function for arbitrary $S \subseteq E$ that only depends on the combinatorial data of $\Rep(\H).$ One  `obvious' choice would be to try $r^\ast_{\Rep(\H)}(S)=\log_b \displaystyle\sum_{\rho \in \Rep(\H)} \displaystyle\prod_{x \in S} \dim \rho_x$ where we write each $\rho \in \Rep(\H)$ as $\displaystyle\bigotimes_{x \in E} \rho_x.$  However, this does not always work. 
  
  \begin{example} (Example \ref{diagonal} continued.)
We continue examining Example \ref{diagonal}. So, $\a=(1,1)$ and $r_\H(\{1\}) = r_\H(\{2\}) = r_\H(\{1,2\})=1.$ Thus, $r^\ast_\H(\{1\}) = r^\ast(\{2\}) = r^\ast_\H(\{1,2\})=1.$ We saw previously that $\Rep(\H)$ was equal to $\{1 \otimes 1, s \otimes s, t \otimes t\},$ where $\{1,s,t\}$ are the three irreducible representations in $\widehat{S_3}.$   In this case the sum of the dimensions of the representations in each single coordinate is $4$, not the expected total of $6$ needed to equal the corresponding rank of the polymatroid $\a$-dual to $r_\H.$   
\end{example}  

Our definition of a rank function based on the combinatorics of $\Rep(\H)$ which is equal to $r^\ast_\H$ is based on the observation that for any $\a$-polymatroid $P=(E,r)$ and $S \subseteq E,~r^\ast(S) = r^\ast(E) - r^\ast_{P/S}(E-S).$   Since the definition makes sense for any representation of $\G,$ we state it in that generality.

\begin{defn}
Let $\Rep$ be a representation of $\G.$ Write $\Rep=\tau_1 \oplus \cdots \oplus \tau_m$ as the direct sum of irreducible representations in $\widehat{\G}.$  Then $r_\Rep:2^E \to \R_{\ge 0}$ is defined by
$$r_\Rep(S)=\log_b \dim \Rep - \log_b \displaystyle\sum_{S \subseteq \triv(\tau_i)} \dim \tau_i.$$ 
\end{defn}

While $r_\Rep$ is always normalized and monotone, it may or may not be submodular.  For example, suppose $\widehat{\Gamma}$ contains a nontrivial one-dimensional irreducible representation $\rho.$ Let $1$ be the trivial representation of $\Gamma,~E=\{1,2,3\}$ and $\Rep= (1 \otimes 1 \otimes \rho) \oplus (1 \otimes 1 \otimes 1) \oplus (\rho \otimes 1 \otimes1).$  Then 
$$r_\Rep (\{1,2\}) + r_\Rep (\{2,3\})= \log_b \frac{9}{4} < \log_b \frac{9}{3} =  r_\Rep(\{1,2,3\}) + r_\Rep(\{2\}).$$
\noindent Nevertheless, $r_{\Rep(\H)}$ does provide a formula for a rank function based on the combinatorics of $\Rep(\H)$ which is equal to $r^\ast_\H.$
\begin{prop}  Let $S \subseteq E.$ Then
$$r^\ast_\H (S) = r_{\Rep(\H)} (S).$$
\end{prop}

\begin{proof}
We start with the definition, $r^\ast_\H(S) = r_\H(E-S) + |\a_S| - r(E).$ Substituting $|\a_S|=|\a|-|\a_{E-S}|$ and reorganizing, 
$$r^\ast_\H(S) =  (|\a|-r(E)) - (|\a_{E-S}| - r_\H(E-S)) = \log_b \bigg(\frac{|\G|}{|\H|}\bigg) - \log_b \bigg(\frac{|\G_{E-S}|}{|\H_{E-S}|}\bigg).$$ 
The dimension of $\Rep(\H)$ is$\log_b \left(\frac{|\G|}{|\H|}\right).$  Now write $\Rep_2(\H) = \tau_1 \oplus \cdots \oplus \tau_m$ as the direct sum of the irreducible representations in $\Rep(\H)$. The extension lemma implies that 
$$\log_b \bigg(\frac{|\G_{E-S}|}{|\H_{E-S}|}\bigg) = \log_b \displaystyle\sum_{S \subseteq \triv(\tau_i)} \dim \tau_i.$$
\end{proof}

  Once one accepts Whitney's suggestion that the dual of a graph may or may not be another graph, it is natural to ask for which graphs $G$ is there a graph $G'$ such that $M[G]^\ast = M[G'].$  This question is answered by Whitney's previous work, \cite[Theorem 29]{Wh2}.  It says that if $G$ is a graph, then there exists a graph $G'$ such that $M[G]^\ast = M[G']$ if and only if $G$ is a planar graph.  With a concrete definition of $r_{\Rep(\H)}$ in hand we can ask which groups, if any, play a role similar to the role planar graphs play with graphs and matroids.  For the remainder of this section $\G= \Gamma^E,~b=|\Gamma|$ and $\a=\a[1].$
  
 \begin{defn}
 A finite group $\Gamma$ is {\bf closed under polymatroid duality} if for all $n$ and subgroups $\H$ of $\G=\Gamma^n,$ there exists $\H'$ a subgroup of $\G$ such that $P^\ast(\H) =P(\H')$
 
 \end{defn} 

\begin{thm}
A finite group $\Gamma$ is closed under polymatroid duality if and only if $\Gamma$ is abelian. 
\end{thm}

\begin{proof}
   Suppose $\Gamma$ is not abelian.  Let $\H$ be the diagonal embedding of $\Gamma$ into $\Gamma^E,~E=\{x,y,z\}.$ Then $P(\H)$ is isomorphic to the matroid $U_{1,3}.$ Hence $P^\ast$ is isomorphic to $U_{2,3}.$ However, $U_{2,3}$ is an excluded minor for $\Gamma,$ hence there is no $\H'$ such that $P(\H')= P^\ast(\H).$
   
   Conversely, assume that $\Gamma$ is abelian.  So all irreducible representations under consideration are one-dimensional and we can choose an isomorphism $\phi: \widehat{\Gamma} \to \Gamma$ and extend it in the obvious way to $\bar{\phi}: \widehat{\G} \to \G.$   By Lemma \ref{normal subgroups}, $\rho \in \widehat{\G}$ is in $\H$ if and only if $\H \subseteq \ker \rho.$  Therefore, $\H'=\bar{\phi}(\Rep(\H))$ is a subgroup of $\Gamma.$ It remains to show that $r_{\H'} = r^\ast_\H.$
   
 Let $S \subseteq E.$ The surjectivity of $\pr_S:\H' \to \H'_S$ and the fact that $\triv(\rho) = I(\bar{\phi}(\rho))$ imply that
 $$|\H'_S| = \frac{|\H'|}{|\ker \pr_S:\H' \to \H'_S|}=\frac{|\G|/|\H|}{|h' \in \H': S \subseteq I(h')|}=\frac{|\G|/|\H|}{|\rho \in \Rep(\H): S \subseteq \triv(\rho)|}$$
 $$=\frac{|\Gamma|^{|E|}/|\H|}{|\G_{E-S}|/|\H_{E-S}|}$$
The last equal sign follows from the extension lemma.  Applying $\log_b=\log_{|\Gamma|}$ to both sides results in 
$$r_{\H'}(S) = |E|-|E-S|-r_\H(E) + r_\H(E-S) = r^\ast_\H(S).$$
\end{proof}

\section{Group codes and the Tutte polynomial } \label{Group codes}

In this section we give a  introduction to group codes and the Tutte polynomial of a polymatroid.  We also give a glimpse at how $\Rep(\H)$ acts as a dual code for a group code $\H$ through the MacWilliams identity for the weight enumerator of a code.  For a close look at how $\Rep(\H)$ gives one possible answer for Dougherty, Kim and Sol\e's open problem, ``Is there a duality and MacWilliams formula for codes over non-Abelian group?'' \cite[Open Question 4.3]{DKS}, which includes MacWiliams identity for the complete weight enumerator, see \cite{We}.  Throughout this section $\G = \Gamma^n$ where $\Gamma$ is a  finite group.

 A {\bf code of length $n$} consists of a finite alphabet $\A$ and a subset $\CC \subseteq \A^n.$  The elements of $\CC$ are called {\bf code words}.   
  When $\A$ is a finite field $\Field$ and $\CC$ a subspace of the $\Field$-vector space $\Field^n$, the code is called an {\bf $\Field$-linear code.}   
   One of the most important invariants of an $\Field$-linear code is its weight enumerator.

Let $\CC$ be an $\Field$-linear code. Given $c \in \CC$ the {\bf weight} of $c$ is $w(c) = |\{i: c_i \neq 0\}|.$ The {\bf weight enumerator} of $\CC$ is 
$$ W_\CC(t) = \displaystyle\sum_{c \in \CC} t^{w(c)}.$$

 The Crapo-Rota critical theorem makes it clear that for an $\Field$-linear code $\CC,$ its matroid $M[\CC]$ determines $W_\CC.$ This was made precise by Greene via the Tutte polynomial of $M[\CC]$ \cite{Gr}. 
 
 The {\bf Tutte polynomial} of a matroid $M$ is a two-variable integer polynomial invariant $T_M(u,v),$ defined by
 
 $$T_M(u,v) = \displaystyle\sum_{ S \subseteq E} (u-1)^{r(E)-r(S)} (v-1)^{|S|-r(S)}.$$
  The applications of the Tutte polynomial  are far too numerous to cover adequately here. For a sampling of applications,  history and theory of the Tutte polynomial  see \cite{EM}.   An alternative definition of the Tutte polynomial is through  deletion/contraction and induction.  
 
 \begin{prop} \label{matroid Tutte recursion} \ 
 \begin{itemize}
  \item If $M$ is the empty matroid, then $T_M(u,v)=1.$
  \item If $x$ is a coloop of $M,$ then $T_M(u,v)=u ~T_{M \setminus x}(u,v).$
  \item If $x$ is a loop of $M,$ then $T_M(u,v) = v ~T_{M/x} (u,v).$
  \item If $x$ is neither a loop nor a coloop, then $T_M(u,v) =T_{M \setminus x}(u,v) + T_{M/x}(u,v).$
 \end{itemize}
 \end{prop}
 
 The proof of this proposition is to split the sum over all subsets $S$ of $E$ into those which do not contain $x$ and those which do contain $x.$ Then first sum can be written in terms of $T_{M \setminus x}(u,v)$ and the latter can be written in terms of $T_{M/x}(u,v).$ The exact same idea allows one to sum up the last three cases into a single equation which can be applied to any $x \in E.$  
 
 \begin{equation} \label{polymatroid Tutte recursion}
 T_M(u,v) = (u-1)^{r(E)-r(E-x)} T_{M \setminus x}(u,v) + (v-1)^{1-r(\{x\})} T_{M/x}(u,v).
 \end{equation}

\noindent We will use this recursive formula for the Tutte polynomial  in the proof of Theorem \ref{Greene for groups} below. 
 
 \begin{thm} \label{Greene's theorem} (Greene's theorem) \cite{Gr}
 Let $\CC$ be an $\Field$-linear code, $q=|\Field|,$ and $d = \dim_\Field \CC.$ Then
 $$W_\CC(t) = (1-t)^d~t^{n-d}~T_{M[\CC]} \bigg(\frac{1+(q-1)t}{1-t}, \frac{1}{t} \bigg).$$
 
 \end{thm}

  There have been several suggested extensions of the Tutte polynomial to integer polymatroids \cite{BKP}, \cite{CMSW}, \cite{OW}.  For subcardinal polymatroids, the simplest way to extend the Tutte polynomial is to use exactly the same definition. Of course, if $P$ is not an integer polymatroid, then $T_P(u,v)$ will not be a polynomial.  In addition, it is not clear how to define $T_P(u,v)$  when $x<1$ or $y<1.$ Nonetheless, Greene's theorem holds without change for group codes. 
  
  A {\bf $\Gamma$-code} of length $n$ is a code whose alphabet is a finite group $\Gamma$ and whose set of codewords  is a subgroup $\H$ of $\G=\Gamma^n.$  The {\bf weight} of a codeword $\gamma=\{\gamma_1, \dots, \gamma_n\}$ is $w(\gamma) = |\{i:\gamma_i \neq 1_\Gamma\}|.$  The {\bf weight enumerator} of a $\Gamma$-code $\H$ is $W_\H(t) = \displaystyle\sum_{h \in \H} t^{w(h)}.$ When $\Gamma$ is the additive group of a finite field $\Field$ and $\H$ is a subspace of the vector space $\Field^n,$ then this definition of the weight of a codeword and  of the weight enumerator of the code are the same as for $\Field$-linear codes.  
 
 \begin{thm} \label{Greene for groups} (Greene's theorem for group codes)
 Let $\Gamma$ be a finite group, $q = |\Gamma|,$ and $\H$ a $\Gamma$-code of length $n \ge 1.$ In addition, set $P(\H)$ to be the polymatroid $(E,r_\H),$ where $E=[n].$ Then,
 \begin{equation} \label{Greene group formula}
 W_\H(t) = (1-t)^{r(E)} t^{n-r(E)} ~T_{P(\H)} \bigg(\frac{1+(q-1)t}{1-t},\frac{1}{t} \bigg)
 \end{equation}
 \end{thm}
 
 The first proof of this theorem was by the third author in \cite{Xu}.  That proof involved interpreting $t$ as a probability.  Since (\ref{Greene group formula}) is not clearly defined for $t \le 0$ or $t \ge 1$ such a proof seems highly appropriate.  Here we give a  deletion/contraction proof using (\ref{polymatroid Tutte recursion}).  
 
 \begin{proof}
 The induction starts with $n=1$ where the definitions show that both sides of (\ref{Greene group formula}) are equal to $|\H|-1.$ For notational convenience we suppress all references to $\H$ when discussing $P(\H)$ and $r_\H.$ We start by determining a deletion/contraction formula for $W_\H(t).$  
 
Let $x \in E$ and $h \in \H.$ If $h_x = 1_\Gamma,$ then $w(h) = w(\pr_{E-\{x\}}(h)).$  Otherwise, $w(h) = w(\pr_{E-\{x\}}(h))+1.$ Since $\pr_{E-\{x\}}$ is surjective, the cardinality of the preimage of every $h' \in \H_{E-\{x\}}$ is 
$$\frac{|\H|}{|\H_{E-\{x\}}|} = |\Gamma|^{r(E) - r(E-\{x\})}= q^{(r(E) - r(E-\{x\})}.$$
\noindent Now,  $K = \{h \in \H: h_x=1_\Gamma\}$ represents the contraction $P/x.$  For every $h \in K,~w(h) = w(\pr_{E-\{x\}}(h)).$ We conclude that 
\begin{equation}
W_\H(t) = q^{r_\H(E) - r_\H(E-\{x\})} \cdot t \cdot W_{\H_{E-\{x\}}}(t) + (1-t)  W_{\H/x}(t). 
\end{equation}

The induction hypothesis, some algebra,  and (\ref{polymatroid Tutte recursion}) imply,
$$
W_\H(t) = q^{r_\H(E) - r_\H(E-\{x\})} \cdot t \cdot t^{n-1-r(E-\{x\})} \cdot (1-t)^{r(E-\{x\})} T_{P(\H) \setminus x} \bigg( \frac{1 + (q-1)t}{1-t}, \frac{1}{t} \bigg)
$$
$$ + $$
$$\bigg(1-t\bigg)\bigg(t^{n-1-(r(E)-r(\{x\})}(1-t)^{r(E)-r(\{x\})}\bigg) T_{P(\H)/x}\bigg(\frac{1 + (q-1)t}{1-t}, \frac{1}{t} \bigg)$$
$$\ $$
$$= t^{n - r(E)} ( 1-t )^{r(E)} \bigg[(\frac{qt}{1-t})^{r(E) - r(E-\{x\})}T_{P(\H) \setminus x} \bigg(\frac{1 + (q-1)t}{1-t}, \frac{1}{t} \bigg) \bigg]$$
$$+$$
$$ t^{n - r(E)} ( 1-t )^{r(E)} \bigg[ (\frac{1-t}{t})^{1-r(\{x\})}  T_{P(\H)/x} \bigg(\frac{1 + (q-1)t}{1-t}, \frac{1}{t} \bigg) \bigg]$$
$$ \ $$
$$= t^{n - r(E)} ( 1-t )^{r(E)} T_{P(\H)} \bigg(\frac{1 + (q-1)t}{1-t}, \frac{1}{t} \bigg)$$

 \end{proof} 
 
 In order to understand the dual relationship between $\H$ and $\Rep(\H)$ we need to establish in what sense $\Rep(\H)$ can be thought of as a code, and what the analog of $W_\H$ is for $\Rep(\H).$  We leave the discussion of how one can  view $\Rep(\H)$ as a code, the proof of Theorem \ref{Dual Greene} below, and an explanation of how $\Rep(\H)$ also satisfies a MacWilliams identity for the complete weight enumerator  for \cite{We}. 
 
  In looking for an interpretation of the weight enumerator of $\Rep(\H)$ the same principle as in Section \ref{Crapo-Rota} applies here. The trivial representation replaces the identity of the group and dimension stands in for counting.
  
  \begin{defn}
  Let $\rho \in \widehat{\G}.$ The {\bf weight} of $\rho$ is $w(\rho) =n- |\triv(\rho)|.$ The {\bf weight enumerator} of $\Rep(\H)$ is
  $$W_{\Rep(\H)}(t) = \displaystyle\sum_{\rho \in \Rep(\H)} (\dim \rho)~t^{w(\rho)}.$$
  \end{defn}
 
 \begin{thm} \label{Dual Greene} \cite{We} (Greene's theorem for $\Rep(\H)$)
 Let $(E,r)$ be the polymatroid $P(\H)$ and $q = |\Gamma|.$ Then
 $$W_{\Rep(\H)}(t) = t^{r(E)} \cdot (1-t)^{n-r(E)} T_{P(\H)} \bigg( \frac{1}{t}, \frac{1 + (q-1)t}{1-t} \bigg).$$
 
 \end{thm}
 
 \begin{proof}
 See \cite{We}.
 \end{proof}

\begin{thm} (MacWilliams identity for groups)
 Let $(E,r)$ be the polymatroid $P(\H)$ and $q = |\Gamma|.$ Then
$$W_{\Rep(\H)}(t) = \frac{(1 + (q-1)t)^n}{q^{r(E)}} W_\H\bigg(\frac{1-t}{1 + (q-1)t}\bigg).$$
\end{thm}

\begin{proof}
The proof is identical to Greene's proof of the MacWilliams identity for linear codes \cite{Gr}.
\end{proof}

\section{Combinatorial Laplacian}

In this last  section we examine the connection between $\Rep(\H)$ and the combinatorial Laplacian.  Throughout this section $E = [n].$ If a polymatroid $(E, r_\H)$ has no elements $x$ such that $r_\H(E) > r_\H(E-x),$ then $\Rep(\H)$ has a concrete realization as the eigenspaces of the combinatorial Laplacian of an associated quotient space described below.  This will lead us to relate the eigenvalues of the combinatorial Laplacian of the top dimension of these quotient spaces to the code weight enumerator of $\H$ viewed as a code with alphabet $\Gamma.$  This result was originally proved for binary matroids by the third author, and provides another example of how $\Rep(\H)$ behaves like Whitney's subspace realization of the dual of a nonplanar graph \cite{Xu}.   We finish the section by using the combinatorial Laplacian to establish  a topological interpretation of the Crapo-Rota critical theorem for finite groups.

The combinatorial Laplacian can be  defined for any (locally finite) regular CW complex.  We will work with regular CW complexes all of whose {\it closed} cells are simplices.  Since Stanley's paper,  \cite{St},  these types of complexes are usually called {\bf simplicial posets} in the combinatorics literature.  The name is derived from the fact that in the face poset  of the complex every lower interval is isomorphic to the face poset of a simplex.  Other names for simplicial posets include Boolean cell complexes and semi-simplicial complexes.  Independent of the name, the reader will not go wrong in thinking of them as  simplicial complexes where there may be several faces with the same vertex set.  A bigon, two vertices with two parallel edges,  is the smallest possible example of a simplicial poset which is not a simplicial complex.  

Let $X$ be a finite simplicial poset. We use  $C_j(X)$ for the simplicial $j$-chains of $X$ with complex coefficients. Equivalently, we denote the $j$-dimensional faces of $X$ by $F_j(X)$ and $C_j(X)=\C[F_j(X)].$  When a specific identification $u \leftrightarrow u^\ast$ between $C_j(X)$ and its dual $C_j(X)^\ast$ is needed, we use the one implied by declaring the preferred basis $\{w_\sigma: \sigma \in F_j(X)\}$ an orthonormal basis of $C_j(X).$ As usual, $\partial_j(X): C_j(X) \to C_{j-1}(X)$ denotes the boundary map.    The empty set is a face of dimension minus one, so $C_{-1}(X) = \C[\emptyset]$ and $C_j=\{0\}$ for all $j < -1$ and $j > \dim X.$ For all $j$ we define 
$\Delta_j(X): C_j(X) \to C_j(X)$ by
$$\Delta_j(X) = \partial^t_j(X) \circ \partial_j(X) + \partial_{j+1}(X) \circ \partial^t_{j+1}(X),$$
where $t$ indicates the transpose.  

The  $\Delta_j(X)$ are collectively known as the {\bf combinatorial Laplacian} of $X.$ The combinatorial Laplacian has several properties.  Among them: 

\begin{itemize}
\item Each $\Delta_j(X)$ is  diagonalizable with nonnegative real eigenvalues.
\item $\ker \Delta_j(X) \cong H^j(X; \C) \cong H_j(X;\C).$ (Hodge theorem)
\item It is possible to approximate the usual Laplacian on a compact Riemannian manifold by using the combinatorial Laplacian on an appropriate sequence of triangulations of the maniold \cite{DP}.
\end{itemize}

Several authors have studied situations in which the eigenvalues of the combinatorial Laplacian are integers.  These include, chessboard complexes \cite{FH}, matroid independence complexes \cite{KRS}, matching complexes \cite{DW}, shifted complexes \cite{DR}.  Representation theory has figured in some of these results (\cite{DW}, \cite{FH}). 
 
 In order to describe the quotient spaces associated to $\H \le \G = \displaystyle\sum^n_{i=1} \Gamma_i,$ we begin by viewing  every $\Gamma_i$ as a zero-dimensional simplicial complex with $|\Gamma_i|$ vertices.  Let $X$ be the simplicial join,
$$ X = \Gamma_1\star \cdots \star \Gamma_n.$$
The facets of $X$ are easily identified with the elements of $\G$ and the vertices of $X$ are equal to the disjoint union of all of the $\Gamma_i,$ where all of the $\Gamma_i$ are disjoint sets even if they are isomorphic as groups.  In general, the faces of $X$ can be identified with the union of all of the $\G_S$ as $S$ runs over all of the subsets of $[n].$ In particular $|F_j(X)| = \displaystyle\sum_{|S| = j+1} |\G_S|.$

Let $\G$ act on the vertices of $X$ so that $(\gamma_1, \dots, \gamma_n) \cdot v_i = v_i \cdot \gamma_i^{-1},~v_i \in \Gamma_i.$  Extend this action to all the faces of $X.$ For instance, if $\gamma=(\gamma_1, \dots, \gamma_n) \in \G$ and $F=v_1 \star \cdots \star v_n$ is a facet of $X,$ then $\gamma \cdot F$ is $v_1 \cdot \gamma^{-1}_1 \star \cdots \star v_n \gamma^{-1}_n.$ Restricting the $\G$-action to $\H$ gives us a quotient space $X/\H$ which is a simplicial poset.  The faces of $X/\H$ can be identified with the right cosets of $\G_S/\H_S$  where $S$ runs over all subsets of $[n].$ When $\Gamma_i = \Z/2\Z$ for all $i,$ such quotient spaces were studied under the name {\bf binary spherical quotients} by the first author in \cite{Sw-t} and \cite{Sw}.  The first determination of the eigenvalues and eigenvectors of the top dimension of the combinatorial Laplacian for binary spherical quotients  was by the third author in \cite{Xu}.   In this case, $M[\H]$ is a binary matroid with $\H$  a subspace of $(\Z/2\Z)^n$ and $M[\H]^\ast$ is represented by the subspace $\H^\perp$ of vectors in $(\Z/2\Z)^n$ orthogonal to $\H.$

\begin{thm} \label{binary eigenvalues} \cite[Theorem 6]{Xu}
Suppose $M[\H]$ is a binary matroid with no coloops.  Then each $h \in \H^\perp$ contributes an eigenvalue of $2 |Z(h)|$ to the eigenvalues of $\Delta_{n-1}(X/\H).$
\end{thm}

For $h \in \H^\perp,$ its code weight enumerator, $w(h),$ equals $ n-Z(h).$ Hence, for (binary) linear codes $\H,$ the eigenvalues of $\Delta_{n-1}(X/\H)$  carry the same information as the code weight enumerator of $\H.$   The top dimensional eigenvalues of the combinatorial Laplacian of a matroid independence complex determine the homotopy type of the complex \cite{KRS}.  In contrast,  the top dimensional eigenvalues of the combinatorial Laplacian of  $X/\H$ do not determine the homotopy type of $X/\H.$ 

\begin{example}
Consider the following  two binary matrices,
$$A = \begin{bmatrix} 1 & 1 & 0 & 0 & 0 & 0 \\ 0 & 0 & 1 & 1 & 0 & 0\\ 0 & 0 & 0 & 0 & 1 & 1\\ \end{bmatrix}, \ \  B = \begin{bmatrix}  1 & 0 & 0 & 1 & 1 & 1 \\ 0 & 1 & 0 & 1 & 0 & 0\\ 0 & 0 & 1 & 1 & 0 & 0 \end{bmatrix}.$$
Set $\H(A)$ and $\H(B)$ to be their corresponding row spaces in $(\Z/2\Z)^6.$ Let $X$ be the join of $\Z/2\Z$ with itself six times.  Equivalently, $X$ is the boundary of the $5$-dimensional cross polytope.  Theorem \ref{binary eigenvalues} tells us that $\Delta_5(X/\H(A))$ and $\Delta_5(X/\H(B))$ have the same spectrum.  On the other hand, $X/\H(A)$ is homeomorphic to $S^5$ \cite[Theorem 4]{Sw},  and $X/\H(B)$ is homeomorphic to the double suspension of $\R P^3$ \cite[Theorem 6.3.1]{Sw-t}.
\end{example}

What if the binary matroid represented by $\H$ has a coloop?  Suppose $x \in E$ is a coloop of the binary matroid $M[\H].$  Let $\G' = \G_{E-\{x\}}, \H'=\H_{E-\{x\}}$ and set $X'/\H'$ to be the quotient space coming from $\G'$ and $\H'.$ Then $X/\H$ is a cone over $X'/\H'$ \cite[Proposition 3]{Sw} and $\lambda$ is an eigenvalue with multiplicity $m$ for $\Delta_{n-1}(X/\H)$ if and only if $\lambda-1$ is an eigenvalue for $\Delta_{n-2}(X'/\H')$ with multiplicity $m$ \cite[Corollary 4.11]{DR}. In particular, all of the eigenvalues of $\Delta_{n-1}(X/\H)$ are integers.  This is true in much greater generality. 

Let $\J$ be a finite group and $U$ a $\J$-set.     We say that the action of $\J$ is {\bf \ef} if the action of $\J$ on $U$ modulo its kernel is free.  For instance, for all $x \in E,~\H$ acts \ef ly on $\Gamma_x$ by restricting the action to the $x$-coordinate of $h \in \H.$  Let $Y_1, \dots, Y_n$ be finite disjoint $\J$-sets on which $\J$ acts \ef ly.  Think of each $Y_i$ as a discrete subset of vertices and set $Y=Y_1 \star \cdots \star Y_n$ to be the join of all of the $Y_i$ as a simplicial complex.  Again, since $\J$ acts on the vertices of $Y$ it also acts on all of the faces of $Y$ and the quotient space $Y/\J$ is a simplicial poset. 

\begin{thm} \label{integer eigenvalues}
  The eigenvalues of $\Delta_{n-1}(Y/\J)$ are integers.
\end{thm}

\begin{proof} 
  We proceed by induction on $|\J|.$ For the base case, assume that $\J$ is the trivial group.  When $n=1,~Y$ is a finite set of vertices, $C_0(Y) = \C[Y]$ and $C_{-1}(Y)=\C[\emptyset].$ By definition, $\Delta_0(Y)$ is multiplication by a  $|Y| \times |Y|$ matrix of all ones.  Therefore, the eigenvalues of $\Delta_0(Y)$ are $|Y|,$ with multiplicity one, and $0$ with multiplicity $|Y|-1.$ Let $\{w_y\}_{y \in Y}$ be the preferred basis for $\C_0[Y]=\C[Y].$ An eigenvector for the $|Y|$-eigenspace is $\displaystyle\sum_{y \in Y} w_y.$  The eigenvectors for the $0$-eigenspace are all vectors of the form
  $$ \displaystyle\sum_{y \in Y} c_y w_y, ~ \displaystyle\sum_{y \in Y} c_y = 0.$$
  For $n>1,$ we first observe that $C_{n-1}(Y) = C_0(Y_1) \otimes \cdots \otimes C_0(Y_n)$ and  apply \cite[Theorem 4.10]{DR} to see that that if $\{v_i\}^n_{i=1}$ are eigenvectors of $\Delta_0(Y_i)$ with corresponding eigenvalues $\lambda_i,$ then 
  $v_1 \otimes \cdots \otimes v_n$ is an eigenvector of $\Delta_{n-1}(Y)$ with eigenvalue $\lambda_1 + \cdots + \lambda_n.$  There are enough eigenvectors of this form to produce a basis of $C_{n-1}(Y).$  Therefore $\Delta_{n-1}(Y)$ has integral eigenvalues.

  Now suppose $\J$ is not trivial. Let $\J_i$ be the kernel of the $\J$-action on $Y_i.$ All of the $\J_i$ are normal subgroups of $\J.$  The kernel of the $\J$ action on $F_{n-1}(Y)$ is $\displaystyle\cap^n_{i=1} \J_i.$ Since we are only concerned with the quotient space $Y/\J$, we might as well assume that this kernel is trivial and hence $\J$ acts freely on $F_{n-1}(Y).$ We consider two cases.   
  
  One, the $\J$-action on $F_{n-2}(Y)$ is also free.  We can characterize this case by looking at the subgroups $\KK_j = \displaystyle\cap_{i \neq j} \J_i.$ The $\J$-action on $F_{n-2}(Y)$ is free if and only if for all $j,~ \KK_j = \{1_\J\}.$ In this case, let $\pr: F_{n-1}(Y) \to F_{n-1}(Y/\J)$ be the quotient map and extend $\pi$ linearly to $p:C_{n-1}(Y) \to C_{n-1}(Y/\J).$
 Suppose $\sigma$ and $\tau$ are facets of $Y$ and $\sigma'= p(\sigma), \tau'=p(\tau).$    As usual, we let $C^\ast_{n-1}(Y), C^\ast_{n-1}(Y/\J), p^\ast$ and $\Delta^\ast_{n-1}$ denote the various dual spaces and maps.  Since $\Delta_{n-1}$ is symmetric, the matrices representing $\Delta_{n-1}$ and $\Delta^\ast_{n-1}$ are the same.  What is the entry of $\Delta_{n-1}(Y)$ corresponding to row $\sigma$ and column $\tau?$ Along the diagonal every entry is $n.$  If $\sigma \neq \tau,$ then the entry is the number of codimension-one faces $\sigma$ and $\tau$ have in common. When the action of $\H$ on $F_{n-2}$ is free, the exact same statement holds for  the entry corresponding to $\sigma'$ and $\tau' $ in $\Delta_{n-1}(Y/\J).$    This is enough to see that if $u$ is an eigenvector of $\Delta_{n-1}(Y/\J)$ with eigenvalue $\lambda,$ then $p^\ast(u^\ast)$ is an eigenvector of $\Delta^\ast_{n-1}(Y)$ with eigenvalue $\lambda$  which we already know is an integer.
 
Two,  the $\J$-action on $F_{n-2}(Y)$ is not free.  So we can choose $j$ such that $|\KK_j| > 1.$ For notational convenience we assume $j=n.$  Let  $Y'=Y/\KK_n=Y_1 \star \cdots \star Y_{n-1} \star (Y_n/\KK_n) .$  Then $Y/\J = Y'/(\J/\KK_n)$ and the induction hypothesis applied to $\J/\KK_n$ finishes the proof.  

 \end{proof}
 
 \begin{cor}
 The eigenvalues of $\Delta_{n-1}(X/\H)$ are all integers.
 \end{cor}

It is tempting to conjecture that the eigenvalues of $\Delta_i(Y/\J)$ are integers for all $0 \le i \le n-1.$ The following example due to A. Chen shows that even if $Y = X/\H$ this might be already false in codimension one and  $\Delta_0.$  

\begin{example} \cite{Ch}
Let $\Gamma = \Z/6\Z,~n=3$ and $\H$ be the subgroup generated by $(\bar{1},\bar{2},\bar{3})$ and $(\bar{2},\bar{1},\bar{4}).$  Then $X$ has three vertices $v_1, v_2,$ and $v_3.$  There are two edges between $v_1$ and $v_2,$ one edge between $v_1$ and $v_3$ and three edges between $v_2$ and $v_3.  $   The eigenvalues of $\Delta_0(X/\H)$ are $0, 6+\sqrt{3}$ and $6-\sqrt{3}.$
\end{example}

While the eigenvalues of $\Delta_i$ may not be integers for general $X/\H,$ P. Mao proved that the eigenvalues of all of the $\Delta_i$ are integers if $X/\H$ is a binary spherical quotient \cite{Mao}.  This has been extended by Chen to any $\H \le \G$ when $P(\H)$ is matroid \cite{Ch}.

Looking at Theorem \ref{binary eigenvalues} we are led to ask ``For general $\G$ and $\H$ what is the  analog of no coloops for a binary matroid?''  A binary matroid  realized by $\H \le (\Z/2\Z)^n$ has  no coloops if and only if there is no element $h \in \H$ which is all $0$'s except for a single $1.$ An essentially identical statement applies to any matroid represented over a field. In terms of the additive group of the vector space $(\Z/2\Z)^n$ this is the  same as saying  that there is no element $h \in \H$ all of whose coordinates except exactly one are the identity.   A third way of saying the same thing is to observe that these two conditions are equivalent to the action of $\H$ on $F_{n-2}(X)$ being free.  By  (\ref{one more}) all of these conditions are the same as $ r_\H(E-\{x\}) = r_\H(E).$

In contrast to matroids represented over a field, $r_\H(E -\{x\})<r_\H(E)$ does not imply that $r_\H(E) = r_\H(E-\{x\}) + r_\H(\{x\}).$  For instance, if $\alpha: \Gamma \to \Gamma$ is any nontrivial homomorphism which is not an automorphism, then $\H = \{(\gamma, \alpha(\gamma)): \gamma \in \Gamma\} \subseteq \Gamma^2$ has $0 < r_\H(\{2\}) < r_\H(\{1,2\}),$ but $r_\H(\{1,2\})= r_\H(\{1\}),$ so $r_H(\{1,2\}) < r_\H(\{1\}) + r_\H(\{2\}).$
In order to get a more detailed understanding of the  the combinatorial Laplacian on $X/\H$ along the lines of Theorem \ref{binary eigenvalues} we first write down the eigenvalues and eigenspaces of $X$ in terms of the irreducible representations of $\Gamma$ and $\G.$  Identify $\C[\G]$ with $C_{n-1}(X)$ using their usual isomorphisms with $\C[\Gamma_1] \otimes \cdots \otimes \C[\Gamma_n].$

\begin{prop} \label{eigenvalues of Delta(X)}
Let $W$ be a $\G$-invariant subspace of $\C[\G]$ equivalent to an irreducible representation $\rho= \rho_1 \otimes \cdots \otimes \rho_n$ of $\G.$ Then $W$ consists of eigenvectors of $\Delta_{n-1}(X)$ all of which have eigenvalue equal to $\displaystyle\sum_{x \in \triv (\rho)} |\Gamma_x|.$ Since $C_{n-1}(X)$ is the direct sum of such $W$, this  describes all of the eigenvalues and eigenvectors of $\Delta_{n-1}(X).$
\end{prop}

\begin{proof}
First consider $n=1.$ As in the proof of Theorem \ref{integer eigenvalues},  $\Delta_0(X)$ is a $|\Gamma| \times |\Gamma|$ matrix all of whose entries are $1.$ So it has two eigenvalues, $|\Gamma|$ and $0.$ The $|\Gamma|$-eigenspace  is one-dimensional and has eigenvector $\displaystyle\sum_{\gamma \in \Gamma} v_\gamma.$ The zero-eigenspace  is $(n-1)$-dimensional and consists of all vectors in $C_0(X)$ of the form
$$ \displaystyle\sum_{\gamma \in \Gamma} a_\gamma v_\gamma, ~ \displaystyle\sum_{\gamma \in \Gamma} a_\gamma = 0.$$
From this we see that $|\Gamma|$-eigenspace is the invariant subspace of $C_0(X)=\C[\Gamma]$ corresponding to the trivial representation and that the zero-eigenspace is the direct sum of all other irreducible invariant subspaces of the regular representation of $\Gamma.$  For $n>0$ apply this to \cite[Proposition 4.9]{DR}.
\end{proof}

As in the proof of Theorem \ref{integer eigenvalues}, we let $p:C_{n-1}(X) \to C_{n-1}(X/\H)$ be the map induced by the quotient map.   We also let $\ast$ denote the various duals of vector spaces, maps, and elements.

\begin{thm} \label{group eigenvalues}
Suppose $r_\H(E-x) = r_\H(E)$ for all $x \in E.$ Let $u$ be an eigenvector of $\Delta_{n-1}(X/\H)$. Then $\rho= p^\ast(u^\ast)$ is in $\Rep(\H)_1.$ The eigenvalue associated to $u$ is $\displaystyle\sum_{x \in \triv (\rho)} |\Gamma_x|.$
\end{thm}
\noindent When $\Gamma_x = \Z/2\Z$ for all $x \in E$ this is equivalent to Theorem \ref{binary eigenvalues}.\\ 

\begin{proof}
   Since $r_\H(E-x) = r_\H(E)$ for all $x \in E,~\H$ acts freely on $F_{n-1}(X)$ and $F_{n-2}(X).$ Now, with $\H$ playing the role of $\J$ and every $Y_i = \Gamma_i,$ we follow the proof of Theorem \ref{integer eigenvalues} to the point that we know that if $u$ is an eigenvector of $\Delta_{n-1}(X/\H)$ with eigenvalue $\lambda,$ then $p^\ast(u^\ast)$ is an eigenvector of $\Delta_{n-1}(X)$ with eigenvalue $\lambda.$ By definition, the coefficients of $p^\ast(u^\ast)$ are constant on cosets of $\H$ and hence $p^\ast(u^\ast)$ is in $\Rep(\H)_1.$ 
    \end{proof}
  
    \begin{cor}
    Suppose $r_\H(E-x) = r_\H(E)$ for all $x \in E.$ Then $\lambda$ is an eigenvalue of $\Delta_{n-1}(X/\H)$ if and only if there exists $\rho \in \Rep(\H)$ such that $\lambda = \displaystyle\sum_{x \in \triv (\rho)} |\Gamma_x|.$ The multiplicity of $\lambda$ is
    $$\displaystyle\sum_{\rho \in \Rep(\H)} \dim \rho,$$
    
    where the sum is taken over all $\rho \in \Rep(\H)$ such that $\displaystyle\sum_{x \in \triv (\rho)} |\Gamma_x|=\lambda.$

    \end{cor}
    \begin{proof} Since $p:C_{n-1}(X) \to C_{n-1}(X/\H)$ is surjective, $p^\ast$ is injective.  The dimension of $C_{n-1}(X/\H) = |F_{n-1}(X/\H)| = |\G/\H| = \dim \Rep(\H).$
    
    \end{proof}

 \begin{rem} 
   If $r_\H(A \cup \{x\}) = r_\H(A)$ for all $A \subseteq E,~|A| \ge n-i, x \in E, $ then similar arguments determine all the eigenvalues and eigenspaces of $\Delta_{n-i}(X/\H)$ in representation theoretic terms. Thus, under the same conditions,  the spectrum of $\Delta_{n-i}(X/\H)$ only depends on $P(\H).$
   \end{rem}
   
   We end by using  $\Delta_{n-1}(X/\H)$ to interpret the dual Crapo-Rota critical theorem for finite groups as a statement about the topology of $X/\H.$   
   
   \begin{thm} \label{char poly as top dimension}
   The dimension of $H_{n-1}(X/\H; \C)$ is 
   \begin{equation} \label{top Crapo-Rota}
   \displaystyle\sum_{\stackrel{\rho \in \Rep(\H)}{\triv \rho = \emptyset}}\dim \rho = \mathlarger{\chi}_{P(\H)^\ast}(|\Gamma|).
   \end{equation}
   \end{thm}
   
   \begin{proof}
   The  equality in (\ref{top Crapo-Rota}) is the dual Crapo-Rota critical theorem for finite groups with $k=1.$   Here we give an independent topological proof.  
   
   The known formulas for the homology of the join of simplicial complexes show that the only nontrivial reduced homology group of $X$ is $\tilde{H}_{n-1}(X).$ Since $\C$ has characteristic zero, $H_i(X/\H;\C)$ equals the elements of $H_i(X;\C)$ left fixed by the action of $\H$ on $H_i(X;\C)$ \cite[Theorem 2.4, pg. 120]{Br}.   By the Hodge theorem, $H_{n-1}(X;\C)$ is the kernel of $\Delta_{n-1}(X).$  Proposition \ref{eigenvalues of Delta(X)} tells us that this kernel is the direct sum of all of the irreducible representations $\rho$ in the regular representation of $\G$ with $\triv \rho = \emptyset.$  By definition, $\Rep(\H)$ are those $\rho$ left fixed by the $\H$-action.  Hence the dimension of $H_{n-1}(X/\H)$ equals the L.H.S. of (\ref{top Crapo-Rota}).  
   
   To see the last equality we consider the reduced Euler characteristic of $X/\H.$  As the only nontrivial reduced homology group of $X/\H$ is $\tilde{H}_{n-1}(X/\H),$ the reduced Euler characteristic of $X/\H$ is $(-1)^{n-1} \dim H_{n-1}(X/\H;\C).$  On the other hand, we can compute the reduced Euler characteristic by counting faces.  Let $S \subseteq E.$ Define $F_S$ to be  the number of faces  of $X/\H$ corresponding to cosets of the form $\G_S/\H_S.$  Then 
   $$|F_S|=\left\vert \frac{\G_S}{\H_S} \right\vert = \frac{|\Gamma|^{|S|}}{|\H_S|} = |\Gamma|^{|S|-r_\H(S)} = |\Gamma|^{r^\ast_\H(E) - r^\ast_\H(E-S)}.$$
   \noindent
   Thus the reduced Euler characteristic of $X/\H$ is
   $$\displaystyle\sum_{S \subseteq E} (-1)^{|S|-1} |F_S| = \displaystyle\sum_{S \subseteq E} (-1)^{|S|-1} |\Gamma|^{r^\ast_\H(E) - r^\ast_\H(E-S)}$$
   $$ = (-1)^{n-1}\displaystyle\sum_{S \subseteq E} (-1)^{|E-S|} |\Gamma|^{r^\ast_\H(E) - r^\ast_\H(E-S)}
 = (-1)^{n-1} \mathlarger{\chi}_{r^\ast_\H} (|\Gamma|).$$
   
   \end{proof}


\begin{thebibliography}{10}
 
 \bibitem{Br}
G. Bredon,
 \newblock {\em Introduction to compact transformation groups},
 \newblock vol. 46, Pure and Applied Mathematics, 
 \newblock Academic Press, New York and London, 1972.

\bibitem{BKP}
O. Bernardi, T. K\'alm\'an and A. Postnikov,
\newblock Universal Tutte polynomial,
\newblock {\em Adv. in Math.}, {\bf 402} (2022), 108355.

 \bibitem{BO}
 T. Brylawski and J. Oxley,
 \newblock {\em The Tutte polynomial and its applications},
 \newblock in Matroid Applications, Encycl. Math. Appl. 40,
 \newblock Cambridge Univ. Press, Cambridge, 1992, 123—225.
 
 \bibitem{BSZ}
 K. Bauer, D. Sen and P. Zvengrowski,
 \newblock A generalized Goursat lemma,
 \newblock {\em Tatra Mt. Math. Publ.}, {\bf 64} (2015), 1–19.
 
 \bibitem{CCLMZ}
 F. Castillo, Y. Cid-Ruiz, B. Li, J. Monta\~no, and N. Zhang,
 \newblock  When are multidegrees postive?,
 \newblock  {\em Adv. in Math.}, {\bf 374} (2020), 107382 (34 pp.).
 
 \bibitem{Ch}
 A. Chen,
 \newblock Personal communication.
 
 \bibitem{CMSW}
L. Ch\'avez-Lomel\'i, C. Merino, G. R. S\'anchez and G. Whittle,
\newblock A new polynomial for polymatroids,
\newblock Australasian J. of Comb., {\bf 80(3)} (2021), 342—360.

 
 \bibitem{CR}
 H. Crapo and G-C. Rota,
\newblock {\em On the foundations of combinatorial theory:  Combinatorial geometries},
\newblock Preliminary edition, M.I.T. press, Cambridge Mass., 1970. 


 
\bibitem{DKS}
S. Dougherty, J. Kim and P. Sol\'e,
\newblock Open problems in coding theory
\newblock in, {\em Noncommutative rings and their applications},
\newblock  Contemp. Math. {\bf 634}, 2015, 79–99.





\bibitem{DP}
J. Doziuk and V. Patodi,
\newblock Riemannian structures and triangulations of manifolds,
\newblock {\em J. Indian Math. Soc. (N.S.)}, 40: 1 —52, 1976.

\bibitem{DR}
A. Duval and V. Reiner,
\newblock Shifted simplicial complexes are Laplacian integral,
\newblock {\em Trans. Amer. Math. Soc.}, 354: 4313—4344, 2002.

 \bibitem{DW}
 X. Dong and M. Wachs,
 \newblock Combinatorial Laplacian of the matching complex,
 \newblock {\em Elect. J. Comb.}, 9: R17, 2002.
 
 \bibitem{E}
 J. Edmonds,
 \newblock 
Submodular functions, matroids and certain polyhedra,
\newblock in {\em Combinatorial structures and their applications. Proc. Calgary Inter. Conf 1969},
\newblock pp. 69—87.
\newblock Gordon and Breach, NY, 1970.

\bibitem{EL}
 C. Eur and M. Larson, 
\newblock Intersection theory of polymatroids,
\newblock {\em Int. Math. Res. Notices},  https://doi.org/10.1093/imrn/rnad213.

\bibitem{EM}
Eds. J. Ellis-Monaghan and I. Moffatt,
\newblock Handbook of a the Tutte polynomial and related topics,
\newblock CRC Press, Boca Raton, New York, London, 2022. 



\bibitem{FH}
 J.~Friedman and P. Hanlon,
\newblock  On the Betti numbers of chessboard complexes,
\newblock  {\em J. of Alg. Comb.}, 8: 193–203, 1998.



\bibitem{GGMS}
I. Gelfand, M. Goresky, R. MacPherson, and V. Serganova,
\newblock Combinatorial geometries, convex polyhedra, and Schubert cells,
\newblock {\em Adv. in Math.}, {\bf 63} (1987), 301–316.

\bibitem{GGW}
J. Geelen, B. Gerards and G. Whittle,
\newblock Solving Rota's Conjecture,
\newblock {\em Notices of the AMS}, {\bf 61} (7), 2014: 736—743.

\bibitem{Go}
E. Goursat,
\newblock Sur les substitutions orthogonales et les divisions r\'eguli\'eres de l'espace,
\newblock {\em Ann. Sciu. \'Ecole Norm. Sup.}, {\bf 6} (1889), 9 — 102. 

\bibitem{Gr}
C. Greene,
\newblock Weight enumeration and the geometry of linear codes,
\newblock {\em Stud. Appl. Math.}, 55 (1976): 119–126. 

\bibitem{H}
T.Helgason,
\newblock Aspects of the theory of hypermatroids,
\newblock in {\em Hypergraph seminar}, 
\newblock (eds. C. Berge and D. Chaudhuri)
\newblock LNM 411, Springer, Berlin-New York, 1974,
\newblock pp. 191—214.




\bibitem{JMW}
S. Jowett, S. Mo and G. Whittle,
\newblock Connectivity functions and polymatroids,
\newblock Adv. in Appl. Math., {\bf 81} (2016), 1—12.  
  
  
 \bibitem{KMR}
J. Kung, G.-C. Rota and M. Murty,
\newblock On the R\'edei Zeta Function,
\newblock {\em J. of Numb. Theory}, {\bf 12}, 1980, 421—436.


\bibitem{KRS}
W. Kook and V. Reiner and D. Stanton,
\newblock Combinatorial Laplacians of matroid complexes,
\newblock {\em J. Amer. Math. Soc.}, 13 (1):129 — 148, 1999.

\bibitem{Ku}
J. Kung, 
\newblock {\em A source book in matroid theory},
\newblock Birkhauser, Boston, 1986.  




\bibitem{Mao}
 P. Mao,
\newblock {\em Cyclic flats in spectrum of combinatorial Laplacian of binary spherical quotients},
\newblock Senior thesis, Cornell University, 2022.  

\bibitem{McD}
C. McDiarmid,
\newblock Rado's theorem for polymatroids,
\newblock {\em Math. Proc. Camb. Phil. Soc.}, {\bf 78} (1975), 263—281.

\bibitem{Mc}
 F. J. MacWilliams,
 \newblock A theorem on the distribution of weights in a systematic code,
 \newblock {\em Bell System Tech. J.}, {\bf 42} (1963), 79--94.

\bibitem{Na1}
 T. Nakasawa,
 \newblock Zur Axiomatik der linearen Abh\"angigkeit I,
 \newblock {\em Sci. Rep. Toyko Bunrika Daigaku}, {\bf 2} (1935), 235–255.

\bibitem{Na2}
T. Nakasawa,
\newblock Zur Axiomatik der linearen Abh\"angigkeit II,
 \newblock {\em Sci. Rep. Toyko Bunrika Daigaku}, {\bf 3} (1936), 123—136.

\bibitem{Na3}
 T. Nakasawa,
 \newblock Zur Axiomatik der linearen Abh\"angigkeit III,
 \newblock {\em Sci. Rep. Toyko Bunrika Daigaku}, {\bf 3} (1936), 45—69.

\bibitem{OW}
J. Oxley and G. Whittle,
\newblock A characterization of Tutte invariants of $2$-polymatroids,
\newblock {\em J. of Comb. Theory}, Series B,
\newblock {\bf 59} (1993), 210–244. 

\bibitem{OSW}
J. Oxley,  C. Semple and G. Whittle,
\newblock A wheels and whirls theorem for $3$-connected $2$-polymatroids, 
\newblock {\em SIAM J. Disc. Math.}, {\bf30 (1)}: 493–524, 2016.

\bibitem{Ox}
J. Oxley,
\newblock {\em Matroid Theory}, 2nd ed.,
\newblock Oxford Grad. Text. Math. 21,
\newblock Oxford Univ. Press, Oxford, 2011.

\bibitem{Po}
A. Postnikov,
\newblock Permutohedra, Associahedra, and Beyond,
\newblock {\em Int. Math. Res. Not.}, (2009), No. 6, 1026 — 1106.

\bibitem{RS}
I. Rival and M. Stanford,
\newblock Algebraic aspects of partition lattices,
\newblock in {\em Matroid Applications}, 106—122,
\newblock Cambridge University Press, Cambridge, 1992. 

\bibitem{Se}
J. P. Serre,
\newblock {\em Linear representations of finite groups},
\newblock Springer-Verlag, New York, 1977.


\bibitem{St}
R.P. Stanley,
\newblock $f$-vectors and $h$-vectors of simplicial posets,
\newblock {\em J. of Pure and Applied Algebra},  71:319 - 331, 1991.


\bibitem{Sw-t}
E.~Swartz,
\newblock Matroids and quotients of spheres,
\newblock Ph.D. thesis, U. Maryland, 1999.

\bibitem{Sw}
E.~Swartz,
\newblock Matroids and quotients of spheres,
\newblock {\em Math. Z.}, 241: 247 - 269, 2002.



\bibitem{Tu}
W. Tutte,
\newblock A homotopy theorem for matroids, II,
\newblock {\em Trans. Amer. Math. Soc.}, {\bf 88} (1958), 161—174.

\bibitem{We}
P. Wentworth-Nice,
\newblock A duality for nonabelian group codes,
\newblock in preparation.

\bibitem{We2}
P. Wentworth-Nice,
\newblock Representability of polymatroids over finite groups,
\newblock in preparation.

\bibitem{Wh2}
H. Whitney,
\newblock Non-separable and planar graphs,
\newblock {\em Trans. Amer. Math. Soc.}, {\bf 34} (1932), 339—362.

\bibitem{Wh}
H. Whitney,
\newblock On the abstract properties of linear dependence,
\newblock {\em Amer. J. Math.}, 57 (3):509–533, 1936.



\bibitem{Whitt}
G. Whittle,
\newblock A geometric theory of hypergraph coloring,
\newblock {\em Aequationes Math.}, 43: 45—58, 1992.

\bibitem{Whitt2}
G. Whittle,
\newblock On matroids representable over $GF(3)$ and other fields,
\newblock{\em Trans. Amer. Math. Soc.}, 349 (2): 579—603, 1997.

\bibitem{Whitt3}
G. Whittle,
\newblock Duality in polymatroids and set functions,
\newblock {\em Comb., Prob. and Comp.}, {\bf 1} (1992), 275—280.


\bibitem{Xu}
A. Xue,
\newblock Combinatorial Laplacian on binary spherical quotients,
\newblock Senior thesis, Cornell University, 2021.

\bibitem{Za}
T.~Zaslavsky,
\newblock The ~{M\"{o}bius} function and the characteristic polynomial,
\newblock in N.L. White, editor, {\em Combinatorial geometries}, Cambridge
  Universtiy Press, 1987.
  
 

\end{thebibliography}
\end{document}